\documentclass{tac}
\usepackage{latexsym,verbatim,ifthen,graphicx}
\usepackage[english]{babel}
\usepackage[utf8]{inputenc}
\usepackage{amssymb}
\usepackage{amsmath}
\usepackage{macros}
\usepackage{tikz-cd}
\usepackage{mathtools}
\usepackage{microtype}

\usetikzlibrary{matrix,arrows}
\usepackage[all]{xy}

\usepackage{mathrsfs}
\usepackage[colorlinks=true]{hyperref}
\hypersetup{allcolors=[rgb]{0.1,0.1,0.4}}

\def\exi{\operatorname{Ex}^\infty}
\def\id{\operatorname{id}}
\def\E{\mathcal E}
\def\bTop{{\Top_\ast}}

\newcommand{\antishriek}{\text{\raisebox{\depth}{\textexclamdown}}}
\renewcommand{\S}{{\mathbb S}}

\newcommand{\coend}{\mathsf{CoEnd}}

\newtheoremrm{recollection}{Recollection}
\title{Enriched coalgebras are sometimes comonadic}
\author{Oisín Flynn-Connolly}
\address{Oisín Flynn-Connolly, Leiden Institute of Advanced Computer Science, Leiden University, Leiden, The Netherlands}
\eaddress{oisinflynnconnolly@gmail.com}
\copyrightyear{2025}

\begin{document}
\maketitle
\begin{abstract}
We introduce an enriched notion of a coalgebra over an operad $\P$ in a symmetric monoidal $\V$-category $\C$. When $\C$ is semicartesian and $\P$ is unital, we construct a $\V$-endofunctor on $\C$ associated to $\P$ and give conditions under which it is a $\V$-comonad with co-Eilenberg--Moore $\V$-category isomorphic to the $\V$-category of $\P$-coalgebras in $\C.$ In many cases, this permits computation of $\V$-categories of coalgebras. The key example is $\V = \Top$ and  $\C = (\bTop, \vee)$, where this construction recovers the comonadic 
description of $C_n$-coalgebras of Moreno-Fernández, Wierstra and the present author. We further 
recover one direction of Fox's theorem.
\end{abstract}

\section{Introduction}
The standard theory of (co)algebras over an operad \cite{Leinster04, fresse09, loday12} assumes that both the operad and its (co)algebras live within the same  symmetric monoidal category. However, mathematical structures have recently emerged that sit less well with this picture.

An example of this phenomenon is the notion of a coalgebra over the little $n$-cubes operad $C_n$ introduced in \cite{ginot12, me1}, with related ideas appearing in \cite{arone14, malin25}, which is used to study structures on the higher Hochschild homology and characterise iterated suspensions. 
The coalgebras over operads in question are pointed spaces $X$ equipped with structure morphisms
\begin{equation*}
\label{1steq}
\Delta_r: C_n(r) \times X \to X^{\vee r}.
\end{equation*}
 This definition mixes the two distinct monoidal structures $\vee$ and $\times$, and pointed and unpointed spaces. From a classical categorical perspective, this definition is quite odd as it combines types in a categorically non-standard way. However, this is completely natural from the perspective of \emph{enriched category theory} \cite{kelly05}. In this paper, we develop this connection, with the goal of giving clean categorical foundation to the topic. This in particular applies to \cite[Section 2]{me1} and \cite[Chapter 6]{flynnconnolly:tel-04974557} along with some of the literature on factorization homology \cite{ginot12, ginot15}. This paper has a target audience that may not be familiar with the particulars of enriched category theory, especially in a non-simplicial setting, so therefore we start with a brief review of it.

The first contribution of this paper is a general notion of an enriched coalgebra over an operad (Definition \ref{coalg}) via coendomorphism operads. While the definition works in the setup of any symmetric monoidal $\V$-category, we only obtain a $\V$-category of such objects when $\V$ is Cartesian. This is because the construction of hom-objects depends on the existence of diagonals. We phrase this as a theorem.

\begin{theorem}
     Let $\V$ be a Cartesian monoidal category and $\C$ be a symmetric monoidal $\V$-category. Then for any operad $\P$ in $\V$, the enriched coalgebras of $\P$ in $\C$ form a $\V$-category.
\end{theorem}

The construction of a monad associated to an operad critically depends on the fact that coproducts distribute over tensor products. The situation for comonads associated to operads is more subtle as the dual of this statement generally fails, and one would not expect the comonad to be given by such a simple formula, or indeed the category of $\P$-coalgebras to be comonadic at all. Indeed, there are simple counter-examples to comonadicity in the absence of either certain limits or of well-behaved cotensors\footnote{as a very simple example, consider the discrete symmetric monoidal category $(\mathbb N, +, 0)$. The only cocommutative coalgebra is $\{0\}$. But the inclusion $\{0\} \to \mathbb N$ does not have a right adjoint.}. This asymmetry is well-known to theoretical computer scientists,  who have observed that final coalgebras can form monads as readily  as comonads \cite{power01}, but is perhaps less well appreciated by topologists. Nonetheless, Anel \cite{anel14}, building on \cite{smith}, has constructed comonads via recursive procedures, under certain technical assumptions, for a different notion of coalgebra in the setting of closed monoidal categories. 

A central technical mechanism of \cite{me1} was demonstrating that under the specific topological conditions of that paper, a very tractable comonadic structure emerges. In this paper, we distill the categorical framework underlying it and show that it can be viewed as a formal consequence of a much more general theory of how the tensor product on $\V$ interacts with that on $\C$.

Given an operad $\P$ in $\V$, the unary component $\P(1)$ is naturally a monoid in $\V$. We define the \emph{coaction $\V$-comonad}, which is a $\V$-comonad on $\C$ induced by this monoid structure and the left (ie., the cotensor) action of $\V$ on $\C$ that looks very much like the classical coaction comonad on an ordinary category. Using \emph{indexed limits}, we define a subcomonad  $C_\P$ of this that takes the higher arity operations of the operad into account. With this object we are able to prove the following theorem, which is the main result of our note.
 
\begin{theorem}
\label{mainthm}
    Let $\V$ be a Cartesian monoidal category and $\C$ be a symmetric monoidal, monically projecting $\V$-category, complete and cotensored over $\V$. Further let $\P$ be a unital operad in $\V.$ Then the $\V$-category of counital $\P$-coalgebras in $\C$ is the co-Eilenberg--Moore $\V$-category of a $\V$-comonad $C_\P$. 
\end{theorem}

The key compatibility condition in the above theorem statement is \emph{monically projecting}. While a precise definition will be given later in the paper, it is an abstraction of the fact that the wedge sum $\vee$ is semicartesian and that the map $X\vee Y \hookrightarrow X\times Y$ in pointed spaces is a monomorphism. The main idea of the proof is that the nullary operation in the operad along with the projection maps of the category, allow one to show that the multiplication is controlled by the unary component $\P(1)$. This theorem admits a strengthening, which states that if $\C$ is isomorphically projecting, ie., if the map $X\otimes Y \hookrightarrow X\times Y$ is an isomorphism, then the comonad is precisely the coaction $\V$-comonad induced by $\P(1).$ We conclude the development of our theory by showing that if $\C$ satisfies a strong version of the semicartesian property, then the computation of $C_\P$ can be rendered quite simple.

We conclude by giving a few examples from topology, set theory, logic and theoretical computer science to which our theory applies. The simplest is small semicartesian categories, which are naturally $\Set$-enriched. Our theorem simply says that, if the category is monically projecting, the category of coalgebras over any $\Set$-operad is comonadic. In particular, Cartesian categories are always monically projecting, in this case meaning that $\C(c, X\times Y) \cong \C(c, X) \times  \C(c, Y) $ is an isomorphism, so in the case of commutative operad in $\Set$, the comonad in question can easily be computed to be the identity. It follows that the category of counital cocommutative coalgebras in any Cartesian category is equivalent to the original category, recovering one direction of a well-known result of Fox \cite{fox76}.

Finally, we have two appendices. The first briefly further reviews the background on symmetric monoidal $\V$-categories. The second consists of some routine verifications that certain categories are symmetric monoidal $\V$-categories. The main result is that the category of $\P$-algebras with categorical coproduct is such an example. The point is that it thus admits the construction of an appropriate $\infty$-coendomorphism operad. The results of this paper, and this section in particular, are intended to be used in forthcoming work to construct operations on the higher Hochschild (co)homology of $E_\infty$-algebras \cite{flynn26}.

Although this paper's results are primarily category-theoretical, they highlight the value of a homotopy theory for symmetric monoidal $\V$-categories in calculations related to iterated suspensions and noncommutative geometry. Currently, to the author's knowledge, a fully satisfactory theory of this kind is lacking in both the model-categorical (see \cite[Section 5.4]{gillou}) and $\infty$-categorical settings. We would particularly be interested if such a theory allowed the interpretation of the relationship between the recognition principles for iterated suspensions and iterated loop spaces as some form of Koszul duality.

\subsection{Related work}
Coendomorphism operads in closed symmetric monoidal categories have been considered in various contexts: modules \cite{loday12, legrignou20, roca24}, topological spaces \cite{baez17, me1}, simplicial sets  \cite{ofc26},  globular categories \cite{Kachour15, cheng23}, spectra \cite{arone14, malin25}, in the infinity context \cite{petersen26} and generally \cite{batanin09, anel14}.
A definition of an algebra over an operad using enriched endomorphism operads is given in \cite[Section 3]{fresse09}. This is used to study derived $A_{\infty}$-algebras in \cite{cirici18} and the derived Deligne conjecture in an unpublished PhD thesis \cite{Aguilar24}. It seems to have been considered in special cases in \cite{johnson22, yau, chabertier24, Espalungue23}. A detailed study of symmetric monoidal $\V$-categories can be found in \cite{Villoria24}.

A line of work developed in depth by Adámek, Milius, Moss and collaborators \cite{Adamek_Milius_Moss_2025, power01, Aczel89}, studies categories of coalgebras for endofunctors on accessible categories, with the goal of establishing the existence of, and performing calculations with, terminal coalgebras. The standard conditions in this literature, eg., accessibility and presentability, are rather different from the monic projection we use; the operadic and mixed-monoidal-structure setting we treat falls outside the scope of those theorems, though we suspect connections may be made in future, especially in light of the recursive constructions of the unpublished \cite{anel14}, which we discuss next.

Anel constructs cofree
$\P$-coalgebra, for a distinct notion of coalgebra defined without using coendomorphism operads, comonads $\P^{\vee}$ for operads $\P$ in symmetric monoidal
categories $\V$ satisfying two technical conditions: the canonical map
$[A, A'] \otimes [B, B'] \to [A \otimes B, A' \otimes B']$ is a monomorphism,
and $\otimes$ commutes with countable intersections in each variable. Anel's construction works by first computing a \emph{lax comonad} and then constructing its comonadic coreflection $\P^{\vee}$ via a transfinite
limit.

Our framework is complementary to this. In Anel's setting, the operad and its coalgebras live in the same category $\V$.
In contrast, we treat operads $\P$ in a Cartesian base $\V$ and coalgebras in a separate $\V$-category $\C$, with its own monoidal structure $\otimes$. For example, the setting of \cite{me1}, where the wedge sum which is not closed,
does not fit Anel's framework.
The compatibility condition between the two monoidal structures in our framework is monic projection, a condition that has no analogue in Anel's setting since he works with a
single monoidal structure. We compute our comonad $C_{\P}$ very explicitly in a single step as a subcomonad of the $\P(1)$-coaction comonad, with no recursion required. In the special case $\C = \V$ with $\times = \otimes$ both frameworks simultaneously apply, and our $C_{\P}$ and $\P^{\vee}$ will agree as cofree $\P$-coalgebra comonads by the uniqueness of right adjoints to the forgetful functor.

\subsection{Notation and conventions}
The collection of morphisms between $A,B \in \mathcal C$ is denoted $\mathcal C(A,B)$. If $\C$ is a $\V$-category, the hom-object is denoted  $\uC(A,B) \in \V.$ The category $\Top$ is taken to be the category of compactly generated, locally Hausdorff spaces and continuous maps. This is a Cartesian closed category, see \cite{steenrod67} for more details on this. To compress notation in some parts, we shall make use of ends, see \cite{loregian21} for more information on those. 
\subsubsection*{The structure of this article}
In Section 2, we briefly recall the elementary notions of enriched category theory, concluding with a discussion on $\V$-comonads. In Section 3, we recall operads, and define the coendomorphism operad and enriched coalgebras. Finally, in Section 4 we define the comonad associated to an operad and prove our main results, concluding with examples.

\subsection{Acknowledgements}
The author thanks Henning Basold, Chase Ford, Gr\'{e}gory Ginot, Jos\'{e} Moreno-Fern\'{a}ndez and Thorsten Wißmann for useful conversations and the anonymous reviewer for useful comments.
\section{Enriched category theory}
In this section, we first revise the ideas from enriched category theory that we shall use in the sequel. We conclude with an extended discussion on $\V$-comonads.
\subsection{Recollections on enriched category theory}
\label{sec:enriched_review}
We assume knowledge of basic category theory \cite{riehl16}, but shall briefly recall the necessary elements of enriched category theory. For a comprehensive treatment, we refer the reader to \cite{kelly05}.

In standard category theory, the collection of morphisms between two objects forms a set. In nature, mapping sets often carry additional structure capturing topological, simplicial, or algebraic data. To handle this, we replace hom-sets with \emph{hom-objects} valued in a suitable base category $\V$.  While in the most general case, one needs only a symmetric monoidal category (the monoidal structure is needed for the third and fourth axioms of Definition \ref{def:enrcat}), most of the theory developed in the literature requires $(\V, \times, \unit_{\V})$ to be a \emph{(Bénabou) cosmos}: that is, a closed, symmetric monoidal category that is complete and cocomplete. Examples of such categories abound and include compactly generated, locally Hausdorff topological spaces with the Cartesian product, simplicial sets with the Cartesian product, metric spaces with the $l_1$-monoidal product and chain complexes with the standard tensor product.

\medskip 

\noindent\textbf{$\V$-categories:} With a sufficiently nice base category in hand, one can immediately define $\V$-categories.
\begin{definition}
\label{def:enrcat}
Let $(\V, \times, \unit_{\V})$ be a symmetric monoidal category. A \emph{$\V$-category} $\C$ consists of:
\begin{enumerate}
    \item A collection of objects $\operatorname{Ob}(\C)$.
    \item For every pair of objects $X, Y \in \C$, a \emph{mapping object} $\uC(X, Y) \in \V$.
    \item For every triple of objects $X, Y, Z \in \C$, a \emph{composition morphism} in $\V$:
    $$ c_{X,Y,Z}: \uC(Y, Z) \times \uC(X, Y) \to \uC(X, Z). $$
    \item For every object $X \in \C$, an identity-assigning morphism in $\V$:
    $$ j_X: \unit_{\V} \to \uC(X, X). $$
\end{enumerate}
\end{definition}
Locally small categories are the same thing as $\Set$-enriched categories.
\begin{example}
    Every ordinary category $\C$ has an associated \emph{free $\V$-category} $\C_{\V}$. This has the same objects, but we have 
    $$
    \underline{\C_{\V}}(X, Y) := \coprod_{f\in \C(X, Y)} \unit_{\V}.
    $$
\end{example}
A \emph{map} $f\from X\to Y$ in $\C$ is shorthand for a map $f\from \unit_{\V} \to \uC(X, Y)$ in $\V$.  Given  a morphism $f: Y\to Z$, we have induced maps
$$
f_\ast: \uC(X, Y) \xrightarrow{\lambda^{-1}} \unit_{\V} \times \uC(X, Y) \xrightarrow{f} \uC(Y, Z)\times \uC(X, Y) \xrightarrow{\circ} \uC(X, Z)
$$
where $\lambda: \unit_{\V} \times \uC(X, Y) \to  \uC(X, Y)$ is the left unitor. Similarly, there is a notion of precomposition $f^\ast$ defined in terms of right unitors. It follows that we can manipulate morphisms as though they are morphisms in an ordinary category and every time we speak of a morphism in $\C,$ this should be implicitly understood.
It immediately follows that a $\V$-category $\C$ has an \emph{underlying category} $\C_0$, which has the same objects and where the morphisms are defined as 
$$
\C_0\left(X , Y\right) \coloneqq \V_0(\unit_{\V}, \uC(X, Y)).
$$
Relatedly, we can also define an enriched notion of \emph{monomorphism}. We say a map $f\from X \to Y$ in $\C$ is an enriched monomorphism if, for every $W\in \C$, the morphism $f_\ast \from \uC(W, X) \to \uC(W,Y)$ is a monomorphism in $\V$.

\medskip

\noindent\textbf{Symmetric monoidal $\V$-categories:}  We will now start building towards a notion of a symmetric monoidal $\V$-category.  First, one has a notion of a $\V$-functor. 
\begin{definition}
A \emph{$\V$-functor} $F: \C \to \D$ between two $\V$-categories assigns to each object $X \in \C$ an object $F(X) \in \D$, and to each pair of objects $X, Y \in \C$ a morphism in $\V$:
$$ F_{X,Y}: \uC(X, Y) \to \uD(F(X), F(Y)) $$
such that $F$ is strictly compatible with the composition and identity morphisms of $\C$ and $\D$.
\end{definition}
The collection of functors between two $\V$-categories $\C$ and $\D$ form a $\V$-category. The hom-object is defined by the end: 
$$
[\C, \D](F, G) \coloneqq \int_{c \in \C} \uD(F(c), G(c)).
$$
Finally, there is an enriched notion of a natural transformation.
\begin{definition}
Let $F, G: \C \to \D$ be two $\V$-functors. A \emph{$\V$-natural transformation} $\alpha: F \Rightarrow G$ consists of an $\operatorname{Ob}(\C)$-indexed family of morphisms in the base category $\V$ from the monoidal unit:
$$ \alpha_X: F(X) \to G(X) \quad \text{for every } X \in \C. $$
\end{definition}
The maps appearing in the definitions above must satisfy various axioms, which are revised in Appendix \ref{appenda}. For now, it suffices to note that we have defined a 2-category $\VCat$. The objects of $\VCat$ are $\V$-categories, the 1-morphisms are $\V$-functors, and the 2-morphisms are $\V$-natural transformations. From the symmetric monoidal structure of $\V$, the category $\VCat$ inherits a symmetric monoidal structure: so we may define the tensor product of $\V$-categories.

\begin{definition}
A \emph{symmetric monoidal $\V$-category} is precisely a symmetric pseudomonoid in the symmetric monoidal 2-category $\VCat$.  More intuitively,  it consists of an object $\C \in \VCat$ equipped with $\V$-functors:
$$ \otimes: \C \otimes \C \to \C \quad \text{and} \quad I: \mathcal{I} \to \C $$
where $\mathcal{I}$ is the unit $\V$-category\footnote{ie., consisting of one object $\ast$ such that $\underline{\mathcal{I}}(\ast, \ast) = \unit_{\V}$.}. It also has the data of invertible 2-morphisms in $\VCat$ for associativity $\alpha$, left (and right)  unitality $\lambda$, and symmetry $\tau$.

\end{definition}
\begin{example}
Our running example throughout this text will be the symmetric monoidal category $(\bTop, \vee)$ of pointed topological spaces with wedge sum. This is a $\Top$-category with Cartesian products $\times$.
\end{example}
Because the monoidal product $\otimes$ is, in particular, a $\V$-bifunctor, for any objects $A, B, C, D \in \C$, there is a morphism in $\V$ that allows us to tensor mapping objects together:
$$ T: \uC(A, C) \times \uC(B, D) \to \uC(A \otimes B, C \otimes D). $$
Note that, rather than denoting it by $\otimes_{(A, B), (C, D)}$, we have chosen to use $T$.  We shall later slightly abuse notation further by also denoting the higher arity variants of $T$ by $T$ also. This morphism $T$ will allow us to define the operadic composition maps for the coendomorphism operad strictly within the enrichment category $\V$.

We say that a symmetric monoidal category is \emph{semicartesian} if the unit of the monoidal structure coincides with the terminal object in the category. For example $(\Top, \times)$ is semicartesian as the terminal object and the unit object are both the one point set $\{\ast\}$. The category $(\bTop, \vee)$ is semicartesian for the same reason. 

In direct analogy, we say that a $\V$-category has a terminal object $1_{\C}$ if $\uC(X, 1_{\C}) \cong \unit_{\V}$ for all $X \in\C.$ We say a symmetric monoidal $\V$-category is \emph{semicartesian} if the terminal object coincides with the unit object\footnote{ie., the object in the image of $I: \mathcal I \to \C$} of the symmetric monoidal structure.

\medskip

\noindent\textbf{Oplax monoidal $\V$-functors:} The last basic notion of enriched category theory that we shall introduce is oplax monoidal $\V$-functors. 
\begin{definition}
    An \emph{oplax monoidal $\V$-functor} $\enr{F} \from \C \to \D$ between symmetric monoidal $\V$-categories is an $\V$-functor along with natural transformations:
    \begin{enumerate}
        \item 
    $$
    \nu \from F(\unit_{\C}) \to  \unit_{\D}
    $$
        \item 
        $$
    \mu_{X, Y} \from F(X\otimes_{\C} Y) \to F(X)\otimes_{\D} F(Y)
    $$
    
    \end{enumerate}
satisfying the enriched analogues of the usual coassociativity and counitality axioms. Importantly, it must also satisfy the identity
$$
\begin{tikzcd}
  \uC(X, Z) \times \uC(Y, Z')\dar{F_{X, Z}\times F_{Y, Z'}} \rar{T} &     \uC(X\otimes_{\C} Z, Y \otimes_{\C} Z') \dar{  (\mu_{FX, FY'})^{\ast}\circ (\mu_{FZ, FZ'})_{\ast} \circ F_{X\otimes Z, Y \otimes_ Z' }}
  \\
  \uD(FX, FZ) \times \uD(FY, FZ') \rar{T} &     \uD(FX\otimes_{\D} FY, FZ \otimes_{\D} FZ')
\end{tikzcd}
$$  
An \emph{oplax monoidal $\V$-functor} is called symmetric if additionally the diagram
$$
\begin{tikzcd}
   F(X\otimes_{\C} Y) \rar{\mu_{X,Y}} \dar{F(\sigma_{X,Y}^{\C})} & F(X)\otimes_{\D} F(Y) \dar{\sigma_{X,Y}^{\D}} 
   \\
    F(Y\otimes_{\C} X) \rar{\mu_{Y,X}} & F(Y)\otimes_{\D} F(X) 
\end{tikzcd}
$$
commutes.
\end{definition}

\medskip

\noindent\textbf{Limits in enriched categories.} Enriched categories possess a more general theory of limits than ordinary ones. Namely, limits $F: \K \to \C$ may be \emph{indexed} by a functor $W: \K \to \V$, where $\K$ is a $\V$-category. We use Kelly's notation for this: $\{W, F\}$. This is defined by the universal property 
$$
\uC(c, \{W, F\}) \cong [\K,\V](W(-), \uC(c, F(-))).
$$
The most basic example of this is the \emph{cotensor} $V\cotens X \in \C$ of an object $X \in \C$ by an object $V \in \V$ is defined by the natural isomorphism in $\V$:
$$\uC(c, V \cotens X) \cong \uV(V, \uC(c, X)).$$ 
\begin{example}
In the running example, the cotensor of $M\in \Top$ and $X\in \bTop$ is the mapping space $\uTop(M, X).$ This contains a unique map that sends the whole domain to the basepoint of $X$. Thus it is pointed itself. It is easy to see that
$$\ubTop(Y, \uTop(M, X)) \cong \uTop(M, \ubTop(Y, X))$$ 
for all $Y\in \bTop$.
\end{example}
More complicated limits can be computed from the cotensor via ends. In general, one has the formula
$$
\{W, F\} \cong \int_{c \in \K} W(c) \cotens F(c).
$$

\subsection{$\V$-comonads and their coalgebras}
\label{def:enrichedcomonad}
Next, we turn to a discussion of comonads in the enriched setting. There are essentially the classical objects defined using the notion of a map in a $\V$-category.
\begin{definition}
Let $\C$ be a $\V$-category. A \emph{$\V$-comonad} $D$ on $\C$ is: 
\begin{enumerate}
    \item a $\V$-endofunctor $D \from \C \to \C$;
    \item a $\V$-natural transformation:
    $$
        \varepsilon_X \from D(X) \to  X
    $$
    (the \emph{counit})
    
    \item a $\V$-natural transformation 
    $$
       \Delta_X \from D(X) \to DD(X)
    $$
    (the \emph{comultiplication})
\end{enumerate}
These are required to satisfy the following axioms:
\begin{enumerate}
    \item  (Coassociativity) One has
    $$
    \begin{tikzcd}
D(X) \arrow[dd, "\Delta_{X}"'] \arrow[rr, "\Delta_{X}"] &  & DDX \arrow[dd, "D(\Delta_X)"] \\
                                                                                                                       &  &                                                                                         \\
DDX \arrow[rr, "\Delta_{DX}"]                                             &  & DDD(X).                                                         
\end{tikzcd}
$$
\item 
(Counitality)
The following diagrams commute: 
$$
\begin{tikzcd}
    DX \dar{\id_{DX}} \rar{\Delta_X}& DDX \arrow[dl, "\varepsilon_{DX}"]
    \\
    DX
\end{tikzcd}
\qquad
\begin{tikzcd}
    DX \dar{\id_{DX}} \rar{\Delta_X}& DDX \arrow[dl, "D(\varepsilon_{X})"]
    \\
    DX
\end{tikzcd}
$$
\end{enumerate}
\end{definition}
A coalgebra over a $\V$-comonad is defined as:
\begin{definition}
    Let $\C$ be a $\V$-category. A coalgebra over a $\V$-comonad $D$ is an object $X\in \C$ along with a map 
    $$
    \gamma: X \to DX
    $$
    satisfying
    \begin{description}
        \item[Coassociativity] The following diagram in $\V$ commutes:
        $$
        \begin{tikzcd}
        X \arrow[r, "\gamma"] \arrow[d, "\gamma"'] & DX \arrow[d, "\Delta_X"] \\
        DX \arrow[r, "D(\gamma)"] & DDX 
        \end{tikzcd}
        $$

        \item[Counitality] The following diagram commutes
        $$
        \begin{tikzcd}
            X \dar{id} \rar{\gamma} & DX  \arrow[dl, "\varepsilon_X"]
            \\
            X
        \end{tikzcd}
        $$
    \end{description}
\end{definition}
The morphisms of coalgebras over a comonad are defined as 
$$
\begin{tikzcd}
X \arrow[rr, "f"] \arrow[dd, "\gamma_X"] &  & Y \arrow[dd, "\gamma_Y"] \\
                                         &  &                          \\
DX \arrow[rr, "Df"]                      &  & DY                      
\end{tikzcd}
$$
To define a $\V$-category $\Dcoalg$ of coalgebras this way, we use an equaliser:
$$
\uDcoalg(X, Y) \coloneqq \eq \begin{tikzcd}
\uC(X, Y) \ar[r,shift left=.75ex,"f"]
  \ar[r,shift right=.75ex,swap,"g"]
&
\uC(X, DY)
\end{tikzcd}
$$
where $f$ is the composite
$$
\uC(X, Y) \xrightarrow{D_{X, Y}} \uC(DX, DY)  \xrightarrow{\gamma_X^{\ast}}\uC(X, DY)
$$
and $g$ is 
$$
\uC(X, Y) \xrightarrow{(\gamma_Y)_\ast}\uC(X, DY)
$$
The equaliser exists as $\V$ is assumed complete. It is a standard argument in  category theory that this defines a  $\V$-category.  For a similar argument, see Theorem \ref{enriefcat}.

\medskip

\noindent\textbf{Example: the coaction $\V$-comonad} Given any monoid in $\V$, the cotensor induces the \emph{coaction $\V$-comonad} on $\C.$

\begin{proposition}
\label{prop: defcoaction}
    Let $(M, m, \nu)$ be a monoid in $\V$. Then there is an induced comonad on $\C$ with underlying endofunctor defined up to isomorphism by
    $$
    X \mapsto M \cotens X.
    $$
\end{proposition}
\begin{proof}
    First we show that one has a $\V$-functor:
    $$
    \uC(X, Y) \xrightarrow{M\cotens(-)_{X, Y}} \uC(M\cotens X, M \cotens Y).
    $$
    To construct this, we first write 
    $$
    \uC(M\cotens X, M \cotens Y)  \cong \uV(M, \uC (M\cotens X,  Y))  
    $$
    It follows by hom-tensor adjunction that it suffices to give a map:
    $$
    \uC(X, Y) \times M \to \uC (M\cotens X,  Y).
    $$
    Finally, completing the construction is the following
    $$
    \uC(X, Y) \times M \xrightarrow{\id_{\uC(X, Y)} \times \nu_X}\uC(X, Y) \times \uC(M\cotens X, X)\xrightarrow{c}\uC (M\cotens X,  Y)
    $$
    where $\nu_X \from M \to \uC(M\cotens X, X)$ is constructed as the following composite:
    $$
    M \times \unit_{\V} \xrightarrow{\id_M\times j_{M \cotens X}} M \times \uC(M \cotens X, M \cotens X) \cong M \times \uV(M, \uC(M \cotens X, X)) \xrightarrow{ev} \uC(M \cotens X, X)
    $$
    where we use the evaluation map.

    Next, we construct the structure maps of the comonad.  The comultiplication is defined as the following morphism:
      \begin{multline*}
    \unit_{\V}   \xrightarrow{j_{M \cotens X}}  \uC\left(M \cotens X, M \cotens X  \right) \cong
    \uV\left(M,  \uC\left(M \cotens X, X  \right) \right) \xrightarrow{\uV\left(m,  \uC\left(M \cotens X, X  \right) \right)} \\\uV\left(M\times M,  \uC\left(M \cotens X, X  \right) \right)  \cong 
    \uV\left(M, \uC\left(M \cotens X, M \cotens X  \right) \right) \cong \\
    \uC(M \cotens X,M \cotens \left( M \cotens X \right)) 
      \end{multline*}
    The counit is constructed as follows:
    \begin{multline*}
        \unit_{\V} \xrightarrow{j_{M\cotens X}} \uC(M\cotens X, M\cotens X) \cong \uV(M, \uC(M \cotens X, X)) \xrightarrow{\uV(\nu, \uC(M \cotens X, X)) } \\
        \uV(\unit_{\V}, \uC(M \cotens X, X)) \cong \uC(M \cotens X, X)
    \end{multline*}
We omit the verification of the coassociativity and the counitality axioms, as it follows by routine diagram chase from the monoid axioms.
\end{proof}

\section{Operads and algebras in symmetric monoidal categories}
\subsection{Recollections on operads}
An operad encodes algebraic structures expressible via the tensor product. Most of the details in this section can be found in \cite{fresse09}. Throughout, operads will be ordinary (i.e., unenriched) operads in $\V$.

Let $(\V, \otimes, \unit_{\V})$ be a symmetric monoidal ordinary category with small colimits. To define an operad, we first need to define the notion of a symmetric sequence. 

\begin{definition}
A \emph{symmetric sequence} in $\V$ is a sequence of objects $A = \{A(n)\}_{n \ge 0}$ in $\V$ such that each $A(n)$ is equipped with a right action of the symmetric group $\mathbb S_n$ ie.\ there is a group homomorphism from $\mathbb S_n$ to $\V(A(n), A(n)).$
\end{definition}

An operad is a symmetric sequence equipped with operations between the objects.

\begin{definition}
An \emph{operad} $\P$ in $\V$ consists of:
\begin{enumerate}
    \item A symmetric sequence $\{\P(n)\}_{n \ge 0}$ in $\V$.
    \item A unit morphism $\eta: \unit_{\V} \to \P(1)$.
    \item For every integer $k \ge 1$ and every sequence of integers $n_1, \dots, n_k \ge 0$, a composition morphism in $\V$:
    \begin{equation*}
        \gamma: \P(k) \times \P(n_1) \times \cdots \times \P(n_k) \to \P(n_1 + \dots + n_k).
    \end{equation*}
\end{enumerate}
These data are required to satisfy the following three axioms:

\begin{description}
    \item[Associativity] 
    The following diagram must commute:
    \begin{center}
    \begin{tikzcd}[column sep=2cm, row sep=1.5cm]
        \P(k) \times \prod\limits_{i=1}^k \P(n_i) \times \prod\limits_{i=1}^k \prod\limits_{j=1}^{n_i} \P(m_{i,j}) 
            \arrow[r, "\id \times \prod \gamma"] 
            \arrow[d, "\gamma \times \id"'] 
        & \P(k) \times \prod\limits_{i=1}^k \P\left(\sum_{j=1}^{n_i} m_{i,j}\right) 
            \arrow[d, "\gamma"] \\
        \P(\sum_{i=1}^k n_i) \times \prod\limits_{i=1}^k \prod\limits_{j=1}^{n_i} \P(m_{i,j}) 
            \arrow[r, "\gamma"'] 
        & \P\left(\sum_{i,j} m_{i,j}\right)
    \end{tikzcd}
    \end{center}

    \item[Unitality] 
    The unit morphism $\eta: \unit_{\V} \to \P(1)$ acts as a two-sided identity for the composition $\gamma$.  The following two diagrams must commute, where $\lambda$ and $\rho$ denote the standard left and right unitors of $\V$:
    \begin{center}
    \begin{tikzcd}[column sep=large, row sep=large]
        \unit_{\V} \times \P(n) 
            \arrow[r, "\eta \times \id"] 
            \arrow[dr, "\lambda"'] 
        & \P(1) \times \P(n) 
            \arrow[d, "\gamma"] 
         \\
        & \P(n) 
        
    \end{tikzcd}
    \end{center}
    \begin{center}
        \begin{tikzcd}
        \P(n) \times \unit_{\V}^{\times n} 
            \arrow[r, "\id \times \eta^{\times n}"', swap] 
            \arrow[dr, "\rho", swap]
            &
            \P(n) \times \P(1)^{\times n}
            \arrow[d, "\gamma"] 
         \\
        & \P(n)  
    \end{tikzcd}
    \end{center}
    
    \item[Equivariance] 
    
    The composition $\gamma$ is compatible with the actions of the symmetric groups.  The group $\mathbb S_k$ has a natural action on  $\mathbb S_{n_1} \times \dots \times \mathbb S_{n_k}$ given by permuting individual blocks. Let $\sigma \in \mathbb S_k$ and $\tau \in \mathbb S_{n_1} \times \dots \times \mathbb S_{n_k}$, where $n = \sum n_i$, then we have a permutation $\sigma \ltimes \tau$, by dividing it into $k$ blocks of with the $i^{th}$ of length $n_i$ and acting on it by $\S_{n_i}$, and permuting the blocks themselves by $\S_k$. Permuting the inputs before composition is equivalent to applying the corresponding block product permutation $\sigma \ltimes \tau \in S_n$, applied to the composed operation.  The following diagram must commute:
    \begin{center}
    \begin{tikzcd}[column sep=3cm, row sep=1.5cm]
        \P(k) \times \P(n_1) \times \cdots \times \P(n_k) 
            \arrow[r, "\sigma \times \tau_{\sigma^{-1}(1)} \times \cdots \times \tau_{\sigma^{-1}(k)}"] 
            \arrow[d, "\gamma"'] 
        & \P(k) \times \P(n_{\sigma^{-1}(1)}) \times \cdots \times \P(n_{\sigma^{-1}(k)}) 
            \arrow[d, "\gamma"] \\
        \P(n) 
            \arrow[r, "\sigma \ltimes \tau"'] 
        & \P(n)
    \end{tikzcd}
    \end{center}
\end{description}
\end{definition}
\begin{remark}
    If we remove the axioms about the symmetric group, \emph{non-symmetric} operads can be defined in an arbitrary monoidal category. 
\end{remark}
\begin{remark}
\label{canonicalrestriction}
Equivalently, Markl has shown that an operad can be defined via \emph{partial composition} operations 
\begin{equation}
    \circ_i\from \P(m) \times \P(n) \to \P(m + n - 1) \quad \text{for } 1 \le i \le m,
\end{equation}
which, informally, represent putting an $n$-ary operation into the $i$-th slot of an $m$-ary operation. These partial compositions satisfy various axioms that can be worked easily out from our previous definition of an operad.
If $\P(0) = \unit_{\V}$, the operad is said to be \emph{unital}. In this case, the $i^{th}$ partial composition map  $d_i\from \P(m) \xrightarrow{\alpha} \P(m) \times \unit_{\V} =  \P(m) \times \P(0) \xrightarrow{\circ_i} \P(m-1)$ induces the \emph{$i^{th}$ canonical restriction operator.}
\end{remark}
\begin{example}
    In any symmetric monoidal category $(\V, \times, \unit_{\V})$ we can define the \emph{commutative operad} $\Com(n) = \unit_{\V}.$ The composition maps of the operad are given in the obvious way by the unitors of the symmetric monoidal category. The symmetric action is given by the obvious action on the factor $\unit_{\V}^{\times n}$ of $\S_n$ in
    $$
    \V(\unit_{\V}^{\times n}, \unit_{\V})\cong \V(\unit_{\V}, \unit_{\V}).
    $$
\end{example}
We conclude with a short discussion about morphisms of operads. The easiest notion is the following: a \emph{morphism of operads} $\psi \from \P \to \Q$ is a collection of maps $\psi_n \from \P(n) \to \Q(n)$ that preserves the identity in $\P(1)$, and is equivariant with respect to the group action and composition. This clearly defines a (ordinary) category of operads $\Op{\V}$ over $\V$. We briefly remark that $\Op{\V}$ can be made a $\V$-category as $\V$ is complete and hom-objects may be constructed using appropriate equalisers - first one constructs morphisms of symmetric sequences and then one restricts to those maps that preserve composition.
\subsection{Enriched coendomorphism operads}
 In \cite{me1}, the authors show that one can define a \textit{coalgebra over an operad} in the symmetric monoidal category of pointed topological spaces endowed with the wedge sum, using the \emph{coendomorphism operad}. Via enriched category theory, this construction can be viewed as an instance of a more general construction. 
 
\begin{definition}
\label{cooo}
Let $(\V, \times, \unit_{\V})$ be a symmetric monoidal category. Let $(\C, \otimes, \unit_\C)$ be a symmetric monoidal $\V$-category. Given $X \in \C$, the coendomorphism operad $\coend(X)$ is an operad in $\V$ with arity $r$ component
\begin{equation*}
    \coend(X)(r) \coloneqq \uC(X, X^{\otimes r}).
\end{equation*}
For $r=0$, set $\coend (X)(0) =\uC(X, \unit_{\C})$.  

\begin{enumerate}
    \item
As $\coend(X)(1) = \uC(X, X)$, we may define the operadic identity to be the identity morphism $j_X\from \unit_\V \to \uC(X, X)$.
\item
By the definition of a symmetric monoidal $\V$-category, there is an induced morphism $T\from \uC(X, X^{\otimes n_1}) \times \cdots \times \uC(X, X^{\otimes n_r}) \to \uC(X^{\otimes r}, X^{\otimes(n_1 + \cdots + n_r)})$. The operadic composition map $\gamma$ is then defined as the composite in $\V$ given by 
$
\gamma = c \circ (\id_{\uC(X, X^{\otimes r})} \times T),
$
where $c$ is the enriched categorical composition.
\item
Given a permutation $\sigma \in \mathbb{S}_k$, the symmetry in the symmetric monoidal structure of the  $\V$-category $\C$ provides a canonical structural isomorphism $\sigma^X\from X^{\otimes k} \xrightarrow{\sim} X^{\otimes k}$. The $\mathbb{S}_k$-action on the arity $k$ component is defined as the $\V$-morphism $\sigma_\ast = \uC(X, \sigma^X) \from \uC(X, X^{\otimes k})\to \uC(X, X^{\otimes k})$ induced by enriched post-composition.
\end{enumerate}
\end{definition}
The coendomorphism operad $\coend(X)$ above depends both on the choice of enrichment $\V$ and the choice of monoidal product on $\C$. We prove that this defines an operad in $\V.$ This is a straightforward verification.
\begin{proposition}
    Let $\C$ be a symmetric monoidal $\V$-category. The coendomorphism operad $\coend(X)$ is a symmetric operad in $\V$.
\end{proposition}
The proof is a straightforward direct verification of the three axioms.
\begin{example}
\label{ourex}
    In \cite{me1}, the coendomorphism operad is an example of this construction: the category $(\bTop, \vee)$ with wedge sum is enriched over the category $\Top$ with Cartesian products $\times$. We have already noted in the introduction that enriched categories provide a more natural perspective on \cite[Remark 2.1]{me1}.
\end{example}
\begin{proposition}
\label{prop2}
    The $\coend(-)$-construction is functorial with respect to isomorphisms in $\C$ and every oplax symmetric monoidal $\V$-functor $(F, \mu, \nu)$ induces a natural transformation of functors
    $$
    \tilde{F} \from \coend(-) \to \coend(F(-)).
    $$
\end{proposition}
\begin{proof}
    We first show functoriality with respect to isomorphisms. Let $\C$ be a symmetric monoidal $\V$-category and let $f\from X \xrightarrow{\sim} Y$ be an isomorphism in $\C$ with given inverse $f^{-1}$. For each arity $r \geq 0$, we define a map in $\V$:
    $$
    \Phi_f(r)\from \coend(X)(r) \to \coend(Y)(r) \qquad  \uC(f^{-1}, f^{\otimes r})(r)\from \uC(X, X^{\otimes r}) \to \uC(Y, Y^{\otimes r}).
    $$
 
   This is compatible with the operadic composition $\gamma$, essentially as one has $(f^{-1})^{\otimes r} \cong (f^{\otimes r})^{-1}$ cancel by the functoriality of the tensor product.  Thus $\Phi_f$ is an isomorphism of operads.

    Next, we show that the construction commutes with oplax monoidal functors. Let $F\from \C \to \D$ be an oplax symmetric monoidal $\V$-functor between symmetric monoidal $\V$-categories. By definition, $F$ comes equipped with natural structural morphisms $\mu_{A,B} \from F(A \otimes B) \to F(A) \otimes F(B)$, which satisfy the obvious naturality conditions. By iterating these maps, we obtain a canonical structural map in $\D$ for any arity $r$:
    $$
    \left( \mu^{(r)} \right)_\ast \from \uD( -, F(X^{\otimes r})) \to \uD( -, F(X)^{\otimes r}).
    $$
    Applying the $\V$-functor $F$ to the mapping objects of $\C$ yields a canonical $\V$-morphism:
    $$
    F_{X, X^{\otimes r}}\from \uC(X, X^{\otimes r}) \to \uD(F(X), F(X^{\otimes r})).
    $$
    We define the induced map of operads $\Psi\from \coend(X) \to \coend_\D(F(X))$ in arity $r$ by post-composing $F_{X, X^{\otimes r}}$ with the iterated oplax structure map $\mu^{(r)}$:
    $$
    \Psi(r) = \left( \mu^{(r)}\right)_\ast \circ F_{X, X^{\otimes r}}\from \uC(X, X^{\otimes r}) \to \uD(F(X), F(X)^{\otimes r}).
    $$
    This defines the desired map of symmetric sequences.

    We must also verify compatibility with the unit. This is easy as 
    $$
    \unit_{\V} \to \uC(X, X) \xrightarrow{F_{X,X}} \uD(F(X), F(X))
    $$
    is the identity as  $\V$-functors send identities to identities.
    Compatibility with the equivariant action follows directly from the fact that the $F$ is symmetric monoidal and therefore commutes with the symmetry maps.
    
    To conclude, we must also verify compatibility with composition.
    The following diagram commutes.  
$$   
\begin{tikzcd}
    \substack{\uC(X, X^{\otimes r}) \times \uC(X, X^{\otimes n_1}) \times \\  \cdots \times  \uC(X, X^{\otimes n_r})} \arrow[d, "\id \times T"] \arrow[r, "F"]  &  \substack{\uD(F(X), F(X^{\otimes r})) \times \uD(F(X), F(X^{\otimes n_1})) \times \\ \cdots \times  \uD(F(X), F(X^{\otimes n_r}))} \dar{(\mu)_{\ast}}
    \\
    \uC(X, X^{\otimes r}) \times  \uC(X^{\otimes r}, X^{\otimes n_1+\cdots+ n_r}) \dar{\circ} & \substack{\uD(F(X), F(X)^{\otimes r}) \times \uD(F(X), F(X)^{\otimes n_1}) \times \\  \cdots \times  \uD(F(X), F(X)^{\otimes n_r})}  \dar{\id \times T}
    \\
    \uC(X, X^{\otimes n_1+\dots + n_r}) \dar{F} &  \substack{\uD(F(X), F(X)^{\otimes r}) \times \\  \uD(F(X)^{\otimes r}, F(X)^{\otimes n_1+\cdots+ n_r})} \dar{\circ}
    \\
    \uD(F(X), F(X^{\otimes n_1+\dots n_r})) \rar{(\mu)_{\ast}} & \uD(F(X), F(X)^{\otimes n_1+\dots n_r})
\end{tikzcd}
$$
In the above diagram, where it is unambiguous, we have abbreviated some of the arrows. But the diagram commutes essentially by  the bifunctoriality axiom of the oplax monoidal functor. Thus, $\Psi$ is a well-defined morphism of operads in $\V$.
\end{proof}
\begin{remark}
Our definition recovers almost all of the coendomorphism operads appearing in the Related Work section of the paper. The only exception is the simplicial coendomorphism operad constructed in \cite{ofc26} which uses the operation $\exi(X \vee X)$, where $\exi$ is Kan's fibrant replacement functor. This does not define a symmetric monoidal product on the nose as it is not strictly associative (but it does define one up to coherent homotopy).
\end{remark}
We also remark in passing that the (co)endomorphism operad construction recovers many well-known examples of operads in the literature. For example, one can consider the subcategory of topological spaces with embeddings as maps. This is enriched over $\Top$ and can be made monoidal by the coproduct $\sqcup.$ Then the enriched endomorphism operad of $D^n$ is equivalent to the framed little $n$-discs operad.
\subsection{Enriched $\P$-coalgebras in a $\V$-category}
Having defined the coendomorphism operad, we may then define a coalgebra in $\C$ over an operad in $\V$. 
\begin{definition}
\label{coalg}
Let $\P$ be a unital operad in $\V$. 
A \textit{$\V$-enriched $\P$-coalgebra} is the data of an object $X \in \C$ and a morphism of operads in $\V$:
\begin{equation*}
    \Phi\from \P \to \coend (X).
\end{equation*}
\end{definition}
\begin{remark}
    We define coalgebras over non-unital operads by also making $\coend (X)$ non-unital by forgetting the $\coend (X)(0)$ component of the operad. 
\end{remark}
\begin{remark}
    A very important observation is that if $\C$ is semicartesian, the coalgebras are canonically counital, because the mapping space $\uC(X, \unit_{\C})$ is the terminal object in $\V,$ so there is only one choice of map $\P(0) \to \coend(X)(0).$
\end{remark}
\begin{example}
     Any locally small category is enriched over $\Set$. It follows that we may consider the action of $\Set$-operads like $\Com$ and $\Ass$ (see \cite{loday12} for the relevant definitions) on any locally small symmetric monoidal category. If we consider vector spaces with the tensor product, this recovers the theory of associative and commutative coalgebras \cite{loday12}.  
 \end{example}
\begin{example}
\label{me1}
In \cite[Theorem A]{me1} it is shown that in the category  $(\bTop, \vee)$, the $n$-fold suspensions $\Sigma^n X$ of a pointed space $X$ has the structure of a coalgebra over the little $n$-cubes operad $C_n$. Observe that this precisely fits this picture: the carriers $\Sigma^n X$ are pointed, and so in $\Top_\ast$. On the other hand, the operad $C_n$ is not pointed and so in the category $\Top$. 
\end{example}
We have not yet defined the $\V$-category of $\P$-coalgebras. For this, we need $\V$ to possess a diagonal map on hom-objects:
$$
\Delta_{X, Y}\from \uC(X, Y) \to \uC(X, Y)^{\times 2}
$$
The easiest way to obtain this is to assume that $\V$ is Cartesian.
\begin{definition}
\label{def:operadicmappingspace}
    Let $(X, \phi_X)$ and $(Y, \phi_Y)$ be $\P$-coalgebras. The $\P$-coalgebra morphism is the equaliser  
    $$
\uPcoalg(X, Y) \coloneqq \eq \begin{tikzcd}
\uC(X, Y) \ar[r,shift left=.75ex,"f"]
  \ar[r,shift right=.75ex,swap,"g"]
&
\prod_{r\geq 0} \uV(\P(r), \uC(X, Y^{\otimes r}) )
\end{tikzcd}
$$
Both maps above are more conveniently expressed by first projecting to the relevant factor in the product and then taking the adjoint.
$$
(f\circ \pi_r)^c ,(g\circ \pi_r)^c \colon \uC(X, Y) \times \P(r) \to \uC(X, Y^{\otimes r})
$$
The maps then are the compositions:
\begin{align*}
 (f\circ \pi_r)^c \colon \uC(X, Y) \times \P(r) 
  &\xrightarrow{\id \times \Phi_X} \uC(X, Y) \times \uC(X, X^{\otimes r}) \\
  &\xrightarrow{\Delta^{(r)} \times \id} \uC(X, Y)^{\times r} \times \uC(X, X^{\otimes r}) \\
  &\xrightarrow{T \times \id} \uC(X^{\otimes r}, Y^{\otimes r}) \times \uC(X, X^{\otimes r}) \\
  &\xrightarrow{c} \uC(X, Y^{\otimes r}) \\[1.5em]
 (g\circ \pi_r)^c \colon  \uC(X, Y) \times \P(r) 
  &\xrightarrow{\id \times \phi_Y} \uC(X, Y) \times \uC(Y, Y^{\otimes r}) \\
  &\xrightarrow{c} \uC(X, Y^{\otimes r})
\end{align*}
\end{definition}

\begin{remark}
When $\V$ is a model category, for example $\Top$ with the Quillen model structure, we can also consider the $\V$-category of derived hom-objects by taking the homotopy equaliser, rather than the regular equaliser in the above. With the definition as above, one would not expect weakly equivalent operads in the Berger--Moerdijk model structure \cite{berger03} to have equivalent homotopy categories and, in fact, generally they will not. For example, the associative operad has non-trivial coalgebras in spaces, while all suspensions are coalgebras over the little 1-cubes operad \cite{me1}.
\end{remark}

\begin{theorem}
\label{enriefcat}
    Let $\V$ be Cartesian and $\P$ be an operad (not necessarily unital). Then, with definitions as above, $\Pcoalg$ is a $\V$-category.
\end{theorem}
\begin{proof}
The base category $\V$ is complete, which ensure that the equaliser defining the mapping objects $\uPcoalg(X, Y)$ exists. To establish that $\Pcoalg$ forms a $\V$-category, we must show that these subobjects admit identity and composition morphisms. 

For any $\P$-coalgebra $X$, it is easy to see that the identity morphism $j_X \colon \unit_\V \to \uC(X, X)$ is strictly equalised by the maps $(f\circ \pi_r)^c$ and $(g\circ \pi_r)^c$.

For composition, we define the restricted map $c'$ as the composite in $\V$:
$$ c' \colon \uPcoalg(Y, Z) \times \uPcoalg(X, Y) \xrightarrow{\eq_{Y,Z} \times \eq_{X,Y}} \uC(Y, Z) \times \uC(X, Y) \xrightarrow{c} \uC(X, Z). $$
To show $c'$ factors uniquely to a composition $c_{\P}$, we must prove it equalises the diagram
$$
\begin{tikzcd}
\uC(X, Z) \ar[r,shift left=.75ex,"f"]
  \ar[r,shift right=.75ex,swap,"g"]
&
\prod_{r\geq 0} \uV(\P(r), \uC(X, Z^{\otimes r}) ) 
\end{tikzcd}
$$
defining $\uPcoalg(X, Z)$. It is again sufficient to check this property on adjoints of the equalisers. The trick is now to use the universal property of the equaliser to move the maps to $Y$.

First, we have the obvious diagram:

$$
\begin{tikzcd}[column sep=huge, row sep=huge]
\uPcoalg(Y, Z) \times \uPcoalg(X, Y) \times \P(r)
    \arrow[r, "\eq_{Y,Z} \times \eq_{X,Y} \times \id"]
    \arrow[d, "\id \times \id \times \phi_Z"']
& \uC(Y, Z) \times \uC(X, Y) \times \P(r)
    \arrow[d, "c \times \id"]
\\
\uPcoalg(Y, Z) \times \uPcoalg(X, Y) \times \uC(Z, Z^{\otimes r})
     \arrow[d, "\eq_{Y,Z} \times \eq_{X,Y} \times \id"]
& \uC(X, Z) \times \P(r)
    \arrow[d, "(g \circ \pi_r)^c"] 
\\
\uC(Y, Z) \times \uC(X, Y) \times \uC(Z, Z^{\otimes r}) \arrow[r, "c \circ (\id\times c)\circ s"]
& \uC(X, Z^{\otimes r})
\end{tikzcd}
$$
The map  $s$ just rearranges factors in the obvious way.

By the defining property of the equaliser of $\eq_{Y,Z}$ we can replace the down-left path with:

$$
\begin{tikzcd}[column sep=huge, row sep=huge]
\uPcoalg(Y, Z) \times \uPcoalg(X, Y) \times \P(r)
    \arrow[r, "\eq_{Y,Z} \times \eq_{X,Y} \times \id"]
    \arrow[d, "\id \times \id \times \phi_Y"']
& \uC(Y, Z) \times \uC(X, Y) \times \P(r)
    \arrow[d, "c \times \id"]
\\
\uPcoalg(Y, Z) \times \uPcoalg(X, Y) \times \uC(Y, Y^{\otimes r})
     \arrow[d, "(\eq_{Y,Z}^{\times r}\circ \Delta^{(r)}) \times \eq_{X,Y} \times \id"]
& \uC(X, Z) \times \P(r)
    \arrow[d, "(g \circ \pi_r)^c"] 
\\
\uC(Y, Z)^{\times r} \times \uC(X, Y) \times \uC(Y, Y^{\otimes r}) 
 \arrow[r, "c \circ (T \times c)\circ s"]
& \uC(X, Z^{\otimes r})
\end{tikzcd}
$$
On the bottom we have once again switched factors. Now, we use the defining property of the equaliser of $\eq_{X,Y}$ to obtain:
$$
\begin{tikzcd}[column sep=huge, row sep=huge]
\uPcoalg(Y, Z) \times \uPcoalg(X, Y) \times \P(r)
    \arrow[r, "\eq_{Y,Z} \times \eq_{X,Y} \times \id"]
    \arrow[d, "\id \times \id \times \Phi_X"']
& \uC(Y, Z) \times \uC(X, Y) \times \P(r)
    \arrow[d, "c \times \id"]
\\
\uPcoalg(Y, Z) \times \uPcoalg(X, Y) \times \uC(X, X^{\otimes r})
     \arrow[d, "(\eq_{Y,Z}^{\times r}\circ \Delta^{(r)}) \times (\eq_{X,Y}^{\times r}\circ \Delta^{(r)} ) \times \id"]
& \uC(X, Z) \times \P(r)
    \arrow[d, "(g \circ \pi_r)^c"] 
\\
\uC(Y, Z)^{\times r} \times \uC(X, Y)^{\times r} \times \uC(X, X^{\otimes r}) 
 \arrow[r, "c \circ (T \times c)\circ s"]
& \uC(X, Z^{\otimes r})
\end{tikzcd}
$$

But now the left down path is precisely equal to $(f \circ \pi_r)^c \circ (c'\times \id)$. Thus, the restricted composition $c'$ strictly equalises the diagram as required.

The associativity and unitality axioms are directly inherited from the underlying $\V$-category $\C$, since composition is a restriction.    
\end{proof}
\section{The comonad associated to an operad}
\label{sec:comonad}

The goal of this section is to give conditions under which the $\V$-category of counital coalgebras over a unital operad $\P$ is comonadic. In this section, we first construct the comonad $C_{\P}$ associated to a unital operad $\P$ in a symmetric monoidal $\V$-category $\C$. The $\V$-category $\C$ should be both complete and cotensored over $\V$.  Most importantly, it should also be semicartesian. We then define \emph{ monically projecting} (Definition \ref{def:monicallyprojecting}), a property of semicartesian $\V$-categories, that relates the tensor product on $\C$ to that on $\V.$ With this setup, we are able to prove the following theorem, which is the main result of this paper.

 \begin{theorem}
 \label{thm:big}
    Let $\V$ be a Cartesian monoidal category and $\C$ be a symmetric monoidal, monically projecting $\V$-category, complete and cotensored over $\V$. Further let $\P$ be a unital operad in $\V.$ Then the $\V$-category of counital $\P$-coalgebras in $\C$ is the co-Eilenberg--Moore $\V$-category of a $\V$-comonad $C_\P$. Moreover if $\C$ is isomorphically projecting the $\V$-comonad $C_\P$ is precisely equal to the $\P(1)$-coaction comonad.
\end{theorem}

We define the comonad object via an indexed limit. We first establish the indexing category. This category contains the information about the arities, the symmetric group actions, and the operadic restriction operations, but not the composition of operations of arity greater than 0.

\begin{definition}
\label{def:K}
Let $\mathcal{K}$ be the small category whose objects are the finite ordinals $\mathbf{n} \ge 0$. The morphisms of $\mathcal{K}(\mathbf{m}, \mathbf{n})$ are generated by the following maps under composition:
\begin{enumerate}
    \item Permutations $\mathbf{n} \to \mathbf{n}.$
    \item The restriction maps $d_i\from \mathbf{n} \to \mathbf{n-1}$ for $1 \le i \le n$, that forget an element.
\end{enumerate}
\end{definition}
The category $\K$ can be viewed as $\mathcal{FI}^{op},$ the opposite of the category of finite sets and injective maps.
\begin{remark}
   \emph{Mutatis mutandis}, the results of this paper extend to non-symmetric operads in general monoidal categories, the only significant change being that one removes the permutations generating the category $\K$ and consider only the order preserving maps. We believe that our results also extend to the  operads with general groups of equivariance considered in \cite{corner14} by further restricting the permutation class.
\end{remark}
\begin{example}
A unital operad $\P$ in $\V$ canonically defines a $\V$-functor $\oP \in  [\mathcal{K},\V]$. The action on $\mathbb{S}_n$ is given by the operadic equivariance, and, in the case of a symmetric operad, the action of $d_i$ is given by plugging the operadic unit $\eta\from \unit_{\V} \to \P(1)$ into the $i$-th input slot of the partial composition. 
\end{example}
A second example of this are tensor products in a symmetric semicartesian monoidal category.
\begin{example}
In a semicartesian monoidal category, the monoidal product $\otimes$ on $\C$ admits natural \emph{collapse maps} $\pi_i\from X^{\otimes n} \to X^{\otimes (n-1)}$. These are defined as follows:
$$
\pi_i\from X^{\otimes n} \xrightarrow{\id^{\otimes i}\otimes t \otimes \id^{n-i-1}} X^{\otimes i} \otimes \unit_{\C} \otimes X^{\otimes n-i-1} \cong X^{\otimes n-1}
$$
where $t$ is the terminal map.
For any object $X \in \C$, the assignment $\mathbf{n} \mapsto X^{\otimes n}$ defines a $\V$-functor $\K_X: \K_{\C} \to \C$, where $\K$ is regarded as a free $\V$-category. The symmetric group acts by permuting the monoidal factors, and the restriction $d_i$ acts via the collapse map $\pi_i\from X^{\otimes n} \to X^{\otimes (n-1)}$.
\end{example}
A semicartesian $\V$-category comes naturally equipped with morphisms
$$
 \uC(c, X \otimes Y) \to \uC(c, X) \times \uC(c, Y)
$$ 
These are built via applying the collapse map separately to each fibre and using the universal property of the product $\uC(c, X) \times \uC(c, Y)$ in $\V$.
\begin{definition}
\label{def:monicallyprojecting}
    We say a symmetric semicartesian $\V$-category $\C$ is \emph{monically projecting} if, for every object $c, X, Y \in \C$, the canonical morphism $\uC(c, X\otimes Y) \to  \uC(c, X)\times\uC(c, Y) $ is a monomorphism in $\V$. If this is an isomorphism, we say that it is \emph{isomorphically projecting}.
\end{definition}
\begin{example}
    The category $(\bTop, \vee)$ over  $(\Top, \times)$ is monically projecting as the map $X\vee Y \to X\times Y$ is a monomorphism.
\end{example}
A simple example illustrating what we are about to do is the following.
\begin{proposition}
\label{prop:monoidcase}
    Let $\C$ be a monically projecting $\V$-category, complete and cotensored over $\V$ and $(M,m, \nu)\in\V$ be a monoid in $\V$. Then the enriched limit 
    $$
    \{\K_M, \K_{(-)}\} \colon \C \to \C
    $$
    viewed as a functor can be equipped with a comonad structure making it isomorphic to the coaction $\V$-comonad.
\end{proposition}
\begin{proof}
    We first show that
    $
    \{\mathcal K_M, \mathcal K_X\}
    $
    is isomorphic to $M\cotens X,$ naturally in $X$. First recall that
    \begin{equation*}
        \uC(c, M\cotens X) \cong \V(M,  \uC(c, X)).
    \end{equation*}
    for all $c\in \C$. We further have that 
    \begin{equation}
    \label{map}
    \V(M,  \uC(c, X)) \cong [\K,\V ](\K_M(-), \uC(c, \K_X(-)))
    \end{equation}
    Note that, by the definition of the functor category, the data of an object
    $$
    [\K,\V ](\K_M(-), \uC(c, \K_X(-)) \cong \int_{n \in \K} \uV \left( M^{\times n}, \uC(c, X^{\otimes n})\right).
    $$
    The right hand side admits limit maps $\pi_i \from  \int_{n \in \K} \uV \left( M^{\times n}, \uC(c, X^{\otimes n})\right) \to \uV \left( M^{\times i}, \uC(c, X^{\otimes i})\right)$ for each $i$.

    The map $\pi_1$ is an isomorphism. To see this, note that it follows from a diagram chase that
    \begin{equation}
    \label{nicemap}
       M^{\times n} \to \uC(c, X^{\otimes n}) \to \uC(c, X)^{\times n}, 
    \end{equation}
    is equal to: 
    $$
    M^{\times n} \xrightarrow{f_1^{\times n}} \uC(c, X)^{\times n}.
    $$
    Now, the final map in (\ref{nicemap}) is a monomorphism. So it follows that $f_1$ uniquely determines $f_n$ as desired.

    The other direction, ie., that every map $f_1\from M\to \uC(c, X)$ determines an element of $ [\K,\V ](\K_M(-), \uC(c, \K_X(-))$ follows by essentially the same argument. 
    
    Proposition \ref{prop: defcoaction} tells us that $M\cotens -$ is a comonad. As it is isomorphic to this, it follows that $\{\K_M, \K_{(-)}\} \colon \C \to \C$ also has the structure of a comonad.
\end{proof}
Now, we do the same thing as the above in the case of a general operad. We define the underlying functor of our prospective comonad as follows.
\begin{definition}
\label{def:comonad_object}
Let $\P$ be a unital operad in $\V$ and $X \in \C$. The $\V$-endofunctor $\CMo \from \C \to \C$ is defined as the indexed limit of the diagram $\mathcal K$ indexed by $\oP$:
$$\CM{X} \coloneqq \{\oP, \mathcal K_X \} . 
$$
\end{definition}
The above definition clearly defines a $\V$-functor as indexed limits are functorial. 
\begin{remark}
\label{rewrite}
    Note that for all $c \in \C$, one has
    $$
    \uC(c, \{\oP, \mathcal K_X \}) = \int_{n \in \K} \uV(\P(n), \uC(c, X^{\otimes n}))
    $$
    where we have used the definition of the hom-objects on $[\K, \V]$. With $c=X,$ this is already highly reminiscent of operad maps into the coendomorphism operad, though we do not yet have compatibility with composition.
\end{remark}
\begin{example}
In the running example, this is an end
$$
\int_{n \in \K} \uTop (\P(n), X^{\vee n}) \cong \prod \uTop (\P(n), X^{\vee n})^{\S_n} / \sim
$$
where we take a big product of mapping spaces, we take the submapping spaces that are invariant under the action of the symmetric group. Finally we have structure induced by the restriction maps, which make the various identifications that are demanded by counitality.  We shall see now that this last step will end up being very strong.
\end{example}
The next theorem shows that in pleasant cases, it can be equipped with a comonad structure. To do this, we shall first establish a link with the coaction $\V$-comonad.
\begin{proposition}
\label{prop:mono}
    Let $\C$ be a monically projecting $\V$-category and $\P$ a unital operad in $\V$. There is a natural monomorphism $\Phi_X \from \CM{X} \to \P(1) \cotens X$. 
\end{proposition}
\begin{proof}
       By Definition \ref{def:comonad_object}, $\CM{X}$ is defined as the $\V$-enriched indexed limit $\{\overline{\P}, \mathcal{K}_X\}$. 
       We observe that 
       $$
       \unit_{\V} \xrightarrow{j_{\CM{X}}} \uC(\CM{X}, \CM{X}) \cong [\K, \V]\left(\oP(-), \uC(\CM{X}, \K_X(-)) \right)
       $$

       By the construction of the $\V$-functor category, this can be expressed as an end in $\V$. More precisely:
$$
[\K, \V]\left(\oP(-), \uC(\CM{X}, \K_X(-)) \right) \cong \int_{n \in \mathcal{K}} \uV(\P(n), \uC(\CM{X}, X^{\otimes n}))
$$
As it is a limit, the end comes equipped with canonical $\V$-natural projection maps to each component. Let $\pi_n(c)$ denote the projection to the $n$-th component:
$$\pi_n(c): \uC(c, \CM{X}) \to \uV(\P(n), \uC(c, X^{\otimes n})).
$$
For $n=1$, this projection gives us a map
\begin{multline*}
\unit_{\V} \xrightarrow{j_{\CM{X}}} \uC(\CM{X}, \CM{X}) \cong [\K, \V]\left(\oP(-), \uC(\CM{X}, \K_X(-)) \right)\\ \xrightarrow{\pi_1(\CM{X})}\uV(\P(1), \uC(\CM{X}, X))
\end{multline*}
Finally, by the universal property of the cotensor, we have
$$
\uV(\P(1), \uC(\CM{X}, X)) \cong \uC(\CM{X}, \P(1) \cotens X).
$$
It follows that we have constructed the desired map
$$
\Phi_X\from \CM{X} \to \P(1) \cotens X.
$$ The next step is to show that $\Phi_X$ is a monomorphism in $\C$. By definition, it is sufficient to show that  $(\Phi_X)_\ast\from \uC(c, \CM{X}) \to \uC(c, \P(1)\cotens X) $ is a monomorphism in $\V$ for all $c \in \C$. 

    In the indexing category $\mathcal{K}$, there are $n$ distinct restriction morphisms $n \to 1$ generated by the maps $d_i$. The definition of the end ensures that any map into it must equalise the action of $\overline{\P}$ and $\mathcal{K}_X$ on these morphisms. We obtain from these the following commutative diagram in $\V$:
\[
\begin{tikzcd}[row sep=large, column sep=huge]
    \int_{n \in \K} \uV(\P(n), \uC(c, X^{\otimes n})) \arrow[r, "\pi_n(c)"] \arrow[d, "\pi_1(c)"] 
    & \uV(\P(n), \uC(c, X^{\otimes n})) \arrow[d, "{\uV(\P(n), \kappa_n)}"] \\
    \uV(\P(1), \uC(c, X)) \arrow[r, "\delta^*"] 
    & \uV(\P(n), \uC(c, X)^{\times n})
\end{tikzcd}
\]
We use in the above Remark \ref{rewrite}.

Here, $\kappa_n$ is the canonical map induced by the semicartesian collapse maps, and $\delta^*$ is the map induced by the operadic restrictions $\P(n) \to \P(1)^{\times n}$ composed with the cartesian diagonal. 

We shall show that $\pi_1(c)$ is a monomorphism. Given $f,g \from T \to \int_{n \in \K} \uV(\P(n), \uC(c, X^{\otimes n}))$ suppose that $\pi_1(c) \circ f= \pi_1(c) \circ g.$ Composing with $\delta^\ast$ we obtain the equality
$$
\uV(\P(n), \kappa_n) \circ \pi_n(c) \circ f=\delta^\ast \circ \pi_1(c) \circ f= \delta^\ast \circ \pi_1(c) \circ g = \uV(\P(n), \kappa_n) \circ \pi_n(c) \circ g $$
where the outer equalities come from the commutative square.

Now, since $\C$ is monically projecting, $\kappa_n$ is a monomorphism in $\V$. The functor $\uV(\P(n), -)$ preserves monomorphisms as it is a right adjoint\footnote{and monomorphisms are limits}, making the right-hand vertical map a monomorphism. We conclude that 
$$
\pi_n(c) \circ f  =   \pi_n(c)\circ g
$$
for all $n.$ Since  
$
\int_{n \in \K} \uV(\P(n), \uC(c, X^{\otimes n}))  
$
is a limit, the family of limit maps $\{\pi_n(c)\}_{n\in \K}$ are jointly monic.  Thus $\pi_1(c)$ is a monomorphism.
\end{proof}
To complete the argument, we would like to go the other way, ie., have a procedure for characterizing the $C_{\P}$ comonad within the coaction comonad. In the enriched setting, we have the following notion of the image of a map.
\begin{definition}
Let $f\from A \to B$ be a morphism. A map $g \from c \to B$ is said to be \emph{in the image of $f$} if there exists a factorisation $\tilde{g}\from c \to A$ such that the following diagram commutes:
$$
\begin{tikzcd}
    c \arrow[drr, dotted, bend right, "\tilde{g}", swap] \arrow[rr, "g"] && B 
    \\
    \ && A \arrow[u, "f"]
\end{tikzcd}
$$ 
\end{definition}
 So, we now characterise the maps $g \from \unit_{\V} \to \uC(c, \P(1) \cotens X)$ in the image of $\Phi_X.$ Note that we have a map:
$$
 g_n \from \uC(c, \P(1)\cotens X) \cong  \uV(\P(1), \uC(c, X)) \xrightarrow{\Delta^{(n)}}  \uV(\P(1)^{\times n}, \uC(c, X)^{\times n}) \to \uV(\P(n), \uC(c, X)^{\times n}).
$$
The final map $\P(n) \to \P(1)^{\times n}$ is just the $n$ projection maps induced by the canonical restriction operators and Cartesian diagonal (see Remark \ref{canonicalrestriction}).
\begin{proposition}
\label{inimage}
    A map $g \from c \to \P(1) \cotens X$ is in the image of $\Phi_X$ if and only if for all $n \ge 2$, the induced $n$-ary components $g_n \from \P(n) \to \uC(c, X)^{\times n}$ factor through the monomorphism $\kappa_n(c) \from \uC(c, X^{\otimes n}) \to \uC(c, X)^{\times n}$.
\end{proposition}
\begin{proof}
    We want to show that a map $g \from c \to  \P(1) \cotens X$ factors as $g = (\Phi_X)_\ast \circ \tilde{g}$ for some $\tilde{g} \from c \to \CM{X} $ if and only if $g_n$ factors through $\kappa_n(c)$ for all $n \ge 2$.

    $(\Rightarrow)$ Assume $g$ is in the image of $\Phi_X$. By definition, there exists a morphism $\tilde{g} \from \unit_{\V} \to \uC(c, \CM{X})$ such that $g = (\Phi_X)_\ast \circ \tilde{g}$. By Remark \ref{rewrite}, one has that $\uC(c, \CM{X}) \cong \int_{n \in \K} \uV(\P(n), \uC(c, X^{\otimes n}))$. The end comes equipped with the limit maps 
    $$\pi_n(c) \from \uC(c, \CM{X}) \to \uV(\P(n), \uC(c, X^{\otimes n})).
    $$

    By the definition of the end, any map into it must equalise the action of $\overline{\P}$ and $\mathcal{K}_X$ on the restriction morphisms. This gives the commutative diagram identity:
    $$
    \uV(\P(n), \kappa_n(c)) \circ \pi_n(c) = \delta^\ast \circ \pi_1(c)
    $$
    where $\delta^\ast$ is the map induced by the canonical restriction operators $\P(n) \to \P(1)^{\times n}$.
    
    The map $g_n \from \P(n) \to \uC(c, X)^{\times n}$ is precisely the application of $\delta^\ast$ to $g$. Therefore:
    $$
    g_n = \delta^\ast \circ g = \delta^\ast \circ \pi_1(c) \circ \tilde{g}.
    $$
    Substituting the diagram identity into the right side yields:
    $$
    g_n = \uV(\P(n), \kappa_n(c)) \circ \pi_n(c) \circ \tilde{g}.
    $$
    This explicitly gives a factorization of $g_n$ through $\uV(\P(n), \kappa_n(c))$, meaning it factors through the monomorphism $\kappa_n(c)$.

    $(\Leftarrow)$ Assume that for all $n \ge 2$, the maps $g_n$ factor through the canonical map $\kappa_n(c) \from \uC(c, X^{\otimes n}) \to \uC(c, X)^{\times n}$.  Thus there exists a family of maps $h_n \from \P(n) \to \uC(c, X^{\otimes n})$ such that $g_n = \uV(\P(n), \kappa_n(c)) \circ h_n$.
    
    We must show that this family of maps forms a compatible cone over the diagram indexed by $\K$. Note that the family $g_n$ already form a compatible cone.
    
    Since $\C$ is a monically projecting $\V$-category, each $\kappa_n(c)$ is a monomorphism in $\V$. The functor $\uV(\P(n), -)$ is a right adjoint and therefore preserves monomorphisms so it follows that $\uV(\P(n), \kappa_n(c))$ is a monomorphism. Moreover essentially by construction, the maps $\kappa_n(c)$ are compatible with cone morphisms in the sense that
    $$
    \begin{tikzcd}
     \uC(c, X)^{\times n} \rar{\uC(c,\pi_i)} & \uC(c, X)^{\times (n-1)}
     \\
      \uC(c, X^{\otimes n}) \rar{\pi_i} \uar{\kappa_n(c)} & \uC(c, X^{\otimes (n-1)})\uar{\kappa_{n-1}}
    \end{tikzcd}
    $$
    where we have overloaded notation by denoting the collapse maps of both $\times$ and $\otimes$ with the same letter. The lifts are $\mathbb S_n$-equivariant because $g_n$ and $\kappa_n(c)$ both are, and $\uV(\P (n),\kappa_n(c))$ is a monomorphism.  
    
    The key step is now to note that the factorisations $h_n$ lift through a compatible family of monomorphisms, so they are unique and strictly inherit all the commuting cone compatibilities from the $g_n$ maps. In other words, we have the following diagram:
    $$
    \begin{tikzcd}
        \unit_{\V} \arrow[rr, "g_n"] \arrow[drr, "h_n", swap] && \uV(\P(n), \uC(c, X)^{\times n} )
        \\
        \ && \uV(\P(n), \uC(c, X^{\otimes n})). \arrow[u, "\uV(\P(n){,} \kappa_n(c))", swap]
    \end{tikzcd}
    $$
    
    Therefore, the family $\{h_n\}$ defines a valid cone over the diagram. By the universal property of the end $\int_{n \in \K} \uV(\P(n), \uC(c, X^{\otimes n}))$, this cone induces a unique map $\tilde{g} \from c \to \CM{X}$ such that $\pi_n(c) \circ \tilde{g} = h_n$ for all $n$. 
    
    In particular, for $n=1$, we have $\pi_1(c) \circ \tilde{g} = h_1 = g$. Since $(\Phi_X)_\ast$ is the map induced by $\pi_1(c)$, we conclude $g = (\Phi_X)_\ast \circ \tilde{g}$, meaning $g$ is in the image of $\Phi_X$.
\end{proof}

\begin{corollary}
\label{isolift}
    If $\C$ is isomorphically projecting and $\P$ is a unital operad, the morphism $\Phi_X$ is an isomorphism for all $X \in \C$.
\end{corollary}
\begin{proof}
    If $\C$ is isomorphically projecting then the lifts in the previous theorem clearly always exist as we can just invert the natural isomorphism $\kappa_{n}\from \uC(c, X^{\otimes n})\to \uC(c, X)^{\times n}$ to obtain the required higher lifts. There is nothing to prove as the inverse of a natural transformation is always natural.
\end{proof}
We may now proceed to the proof of the main result.
\begin{theorem}
In any monically projecting $\V$-category $\C$, complete and cotensored over $\V,$  the assignment $X \mapsto \CM{X}$ naturally extends to a $\V$-comonad $(C_{\P}, \varepsilon, \Delta)$ on the  $\V$-category $\C$.
\end{theorem}

\begin{proof}
We shall show that it is a subcomonad of the coaction comonad.

We have already established in Proposition \ref{prop:mono} that there is a natural monomorphism
$$
\Phi_X \from \CM{X} \to \P(1) \cotens X.
$$

The counit is immediately defined by the map $C_{P}(X) \xrightarrow{\Phi_X} \P(1) \cotens X \to X$, where the last map is the counit of the coaction comonad.

Next we shall construct the dotted map in the following diagram
$$
\begin{tikzcd}
    \CM{X} \arrow[d, dotted, "\Delta_X"]\arrow[rrr,"\Phi_X"] &&& \P(1) \cotens X \dar{\Delta^{cw}_{X}}
    \\
    \CM{\CM{X}} \arrow[r, "\Phi_{\CM{X}}"] &  \P(1) \cotens(\CM{X})  \arrow[rr, "\left(\P(1) \cotens(\Phi_X)\right)"] && \P(1) \cotens \left(\P(1) \cotens X\right)
\end{tikzcd}
$$
We expand the above diagram in terms of the unit of $\V.$
$$
\begin{tikzcd}
    \unit_{\V} \arrow[rr, "\Phi_X"] \arrow[d, dotted] && \uC(C_{\P}(X), \P(1) \cotens X) \arrow[dd, "\Delta^{cw}_X"]
    \\
    \uC(\CM{X}, \CM{\CM{X}}) \arrow[d, "\Phi_{\CM{X}}"] &  & \
    \\
     \uC(\CM{X},\P(1) \cotens \CM{X})  \arrow[rr, "\P(1) \cotens \Phi_X"] &&  \uC(\CM{X}, \P(1) \cotens \left( \P(1) \cotens X \right)) 
\end{tikzcd}
$$
Here, the horizontal maps on the top and bottom are monomorphisms. The only non-trivial one to observe is $\left(\P(1) \cotens \Phi_X\right)$, this is a monomorphism because $\P(1) \cotens -$ is a right adjoint and monos are limits. Therefore, if the dotted map exists it is necessarily unique.  It follows that the coassociativity and counitality axioms for $C_{\P}$ are then trivially satisfied. This is because of the factorisation above, any diagram that commutes in the coaction comonad must commute in the subcomonad $C_{\P}$ and so the conclusion follows.

So it suffices to construct the dotted map. To do this, we shall show using Proposition \ref{inimage} that $\Delta^{cw}_X \circ \Phi_X$ is in the image of both $\left(\P(1) \cotens(\Phi_X)\right)$ and $\Phi_{\CM{X}}$.
\begin{enumerate}
    \item \emph{Factoring through $\P(1) \cotens \CM{X}$:} The $\P(1)$-cotensor coaction comonad comultiplication was explicitly constructed in Proposition \ref{prop: defcoaction}. We need to characterise the image of the map $(\P(1) \cotens \Phi_X)_\ast$. To do this, we have that by hom-tensor adjunction:
    $$
    \begin{tikzcd}
    \uC(\CM{X},\P(1) \cotens \CM{X}) \dar{\cong} \arrow[rr, "\left(\P(1) \cotens(\Phi_X)\right)_\ast"] && \uC(\CM{X}, \P(1) \cotens \left( \P(1) \cotens X \right)) 
    \\
    \uV(\P(1), \uC(\CM{X}, \CM{X})) \arrow[rr,  "\uV\left(\P(1){,} (\Phi_X)_\ast\right)"] && \uV(\P(1), \uC(\CM{X}, \P(1)\cotens X)). \uar{\cong}
    \end{tikzcd}
    $$
    Moreover, we can also go back to the proof of  Proposition \ref{prop: defcoaction} and extract that coaction comonad map factors as 
    \begin{multline*}
    \uV(\P(1), \uC(\CM{X}, X)) \xrightarrow{\uV(\gamma, \uC(\CM{X}, X))} \uV(\P(1)\times \P(1), \uC(\CM{X}, X)) \\\cong \uV(\P(1), \uC(\CM{X}, \P(1)\cotens X))
    \end{multline*}
    We know that if $g \from \unit_{\V}\to \uC(\CM{X}, \P(1)\cotens X) \cong \uV(\P(1), \uC(\CM{X}, X)$ is in the image of $\Phi_X,$ the induced maps $g_n \from \P(n) \to \uC(\CM{X}, X)^{\times n}$ factorise through some $\tilde{g_n}\from \P(n) \to \uC(\CM{X}, X^{\otimes n}).$ 
    It follows from the above discussion that there is an induced factorisation 
    $$
    h \from \P(1) \times \P(1) \xrightarrow{\gamma} \P(1) \xrightarrow{g}  \uC(\CM{X}, X)
    $$
    Now, keeping the first copy of $\P(1)$ fixed, we can obtain maps 
    $$
    \P(1)\times \P(n) \xrightarrow{\id\times d_i} \P(1)\times \P(1) \xrightarrow{h} \uC(\CM{X}, X)
    $$
    Putting these all together gives a map
    $$
    \P(1)\times \P(n) \xrightarrow{h_n} \uC(\CM{X}, X)^{\times n}
    $$
    Because operadic composition is associative (and the canonical restriction maps come from composition), we have the following diagram: 
    $$
    \begin{tikzcd}
        \P(1) \times \P(n) \arrow[d, "\gamma"]\rar{h_n} & \uC(\CM{X}, X)^{\times n}
        \\
        \P(n) \arrow[ur, "g_n"] & \
    \end{tikzcd}
    $$
    In particular, $h_n$ factors through $ \tilde{g_n}\circ \gamma$, proving the result.
    \item \emph{Factoring through $\CM{\CM{X}}$:} 
    This follows essentially by the previous argument.
\end{enumerate}
Because the map satisfies the defining equalisers at both stages, it factors uniquely through the limit $\CM{\CM{X}}$, yielding our well-defined comonad comultiplication $\Delta_X$.
\end{proof}
 
We may now prove our main result.

\begin{proof}[of Theorem \ref{thm:big}]
The isomorphically projecting claim follows directly from Corollary \ref{isolift}.

    Let $X$ be a $\P$-coalgebra. Then there is an operad morphism
    $
    \P  \to \coend(X).
    $
    In particular, one has a collection of $\V$-morphisms
    $
    \P(n) \to \uC(X, X^{\otimes n}).
    $ 
    By hom-tensor adjunction, this defines a $\V$-natural family of maps
    $
    \unit_{\V} \to\uC(X, \P(n) \cotens X^{\otimes n} ).
    $
The compatibility with the operad maps guarantees, in particular, compatibility with the morphisms in $\K$ and so one has constructed the desired map $\unit_{\V} \to\uC(X, \CM{X})$.

The various identities are now easily checked. The key verification is that the two definitions (Definition \ref{def:operadicmappingspace} and that in Section \ref{def:enrichedcomonad}) of mapping spaces agree. The comonadic mapping space is the equaliser:
$$
\uCPcoalg(X, Y) \coloneqq \eq \begin{tikzcd}
\uC(X, Y) \ar[r,shift left=.75ex,"f"]
  \ar[r,shift right=.75ex,swap,"g"]
&
\uC(X, \CM{Y})
\end{tikzcd}
$$
where
\begin{align*}
&f\from \uC(X, Y) \xrightarrow{(\CMo)_{X, Y}} \uC(\CM{X}, \CM{Y})  \xrightarrow{\gamma_X^{\ast}}\uC(X, \CM{Y})
\\
&g \from \uC(X, Y) \xrightarrow{(\gamma_Y)_\ast}\uC(X, \CM{Y}).
\end{align*}
Note that $\uC(X, \CM{Y}) \cong \int_{n \in K} \uV(\P(n), \uC(X, Y^{\otimes n}))$ by Remark \ref{rewrite}. Since $\V$ is Cartesian, the limit maps fit together to define a map 
$$
\int_{n \in K}  \uV(\P(n), \uC(X, Y^{\otimes n}) \to \prod_{n\geq 1} \uV(\P(n), \uC(X, Y^{\otimes n})).
$$
Moreover, by a simple diagram chase, $f$ and $g$ can easily be seen to be equal to the $f$ and $g$ appearing in Definition \ref{def:operadicmappingspace}. It follows that there is an isomorphism of $\V$-categories between operadic and comonadic coalgebras.
\end{proof}
\begin{remark}
Some of the theory of this section is salvaged even when the monically projecting assumption fails. Given any semicartesian $\V$-category and unital operad $\P$, we can look at the subcategory of $\C$ that consist of objects that \emph{monically project} ie.\ such $\uC(c, X^{\otimes n})\to \uC(c, X)^{\times n}$ is a monomorphism. Then, all of our theory applies to this subcategory.
\end{remark}
Corollary \ref{isolift} allows us to explicitly compute the $\V$-comonad when $\C$ is isomorphically projecting. In general, this may rather difficult as, by Proposition \ref{inimage}, one must show that $f_1 \in \P(1)\cotens X$ satisfies infinitely many lifting properties.  It is thus natural to ask for a simpler criterion.
\begin{definition}
   A semicartesian $\V$-category $\C$ is \emph{$k$-strong monically projecting} if  it is monically projecting and, for every object $X \in \C$ and $n \ge k$, the canonical monomorphism $\kappa_n \from X^{\otimes n} \to X^{\times n}$ exhibits $\uC(c, X^{\otimes n})$ as the wide pullback taken over the ${n}\choose{k}$ projections $\pi(c) \from \uC(c, X^{\otimes n}) \to \uC(c, X^{\otimes k})$.
\end{definition}
\begin{example}
    The semicartesian $\Top$-category $(\bTop, \vee)$ is 2-strong. Any isomorphically projecting category is 1-strong. 
\end{example}
The following corollary is easily seen and rather useful: it essentially implies that the arity greater than $k$ component of the operad are redundant in a $k$-strong monically projecting $\V$-category. 
\begin{corollary}
    Let $\C$ be a $k$-strong semicartesian $\V$-category and $\P$ a unital operad. A morphism $g$ is in the image of $\Phi_X$ precisely when the induced $i$-ary components $g_i \from \P(n) \to \uC(c, X)^{\times i}$ factor through the monomorphism $\kappa_i(c) \from \uC(c, X^{\otimes i}) \to \uC(c, X)^{\times i}$ for $i\leq k.$
\end{corollary}
\begin{proof}
    For $j > k$, we have $g_j \from  \P(n) \to \uC(c, X)^{\times j}$. This then factors through the wide pullback of the projections onto the $k$-ary components. But this is then isomorphic to $\uC(c, X^{\otimes i})$, and we invert this isomorphism to obtain the desired map. The result then follows by Proposition \ref{inimage}.
\end{proof}
\begin{example}
    The comonad $C_{C_n}(X)$ associated to the little $n$-cubes operad consists of the subspace of $ f\in \uTop(C_n(1), X)$ such that the map 
    $$
    C_n(2) \xrightarrow{\pi_1 \times \pi_2} C_n(1)\times C_n(1) \xrightarrow{f\times f } X^{\times 2}
    $$
    factors through $X^{\vee 2}$, where $\pi_i \from C_n(2) \to C_n(1)$ is the map that forgets the $i^{th}$-disc. There are no further compatibility relations to verify.
\end{example}
There are many applications that follow fairly immediately from our prior theoretical development. 
\begin{enumerate}
    \item 
    Any small semicartesian category $\C$ satisfies the conditions of Theorem \ref{thm:big}. It follows that given any operad $\P$ in $\Set$, the category of $\P$-coalgebras in $\C$ is comonadic. More concretely, we can consider the operad $\Com$\footnote{the first part of this discussion holds for any operad $\P$ such that $\P(1) = \{\ast\}.$} and a Cartesian category $(\C, \times).$ Here we have 
    $$
    \Com(n) \coloneqq \{\ast\}.
    $$
    the one element set with trivial action of the symmetric group. Then a $\Com$-coalgebra is a cocommutative map:
    $$
    X\to X \times X.
    $$
    along with a canonical counit $X \to \ast.$ In particular, this means that we have projection maps $\pi_1, \pi_2\from X\times X \to X.$ Now, by Theorem \ref{thm:big}, the category of such coalgebras is comonadic over a subcomonad of the $\Com(1)$-comonad. We have that 
    $$
    \Set(\{\ast \}, \C(X, Y)) \cong \C(X, Y).
    $$
    It follows that $\{\ast\} \cotens Y \cong Y,$ so the coaction comonad in question is just the identity comonad. One can further observe that by the definition of the product, in a Cartesian category, $\C(c, X\times Y) \cong \C(c, X) \times \C(c, Y),$ so it further follows that it is isomorphically projecting and therefore the identity comonad. It follows that the category of counital cocommutative coalgebras in $\C$ is precisely $\C,$ and, moreover, the coalgebra structure is induced by the diagonal. This is precisely one of the directions of Fox's theorem \cite{fox76}. 

    The map $\C(c, X\otimes Y) \to \C(c, X) \times \C(c, Y)$ may fail to be an isomorphism in an arbitrary monically projecting semicartesian category. For example, we may consider $(\bTop, \vee)$.  In this case, the comonad will be a subcomonad of the identity comonad.  This implies that every object in the category has either one or zero cocommutative counital coalgebra structures on it. In the case of $(\bTop, \vee)$, the only cocommutative coalgebra is $\ast.$ Another example: Lawvere metric spaces with addition \cite{Lawvere73}. This is not Cartesian, but it can be easily verified to be semicartesian and monically projecting over $\Set$.
    \item 
    To generalise last part of the previous example, let $\P$ be a connected\footnote{meaning an operad $\P$ such that $\P(1)$ is the initial object $i$ in $\V.$}, unital operad. Then, we have that 
    $$
    \uV(i, \uC(X, Y)) \cong  \unit_{\V}.
    $$
    it follows that $i \cotens Y \cong \unit_{\C}.$ So therefore, under the assumptions of Theorem \ref{thm:big} the category of enriched coalgebras over any such  unital operad $\P$ automatically consists of just $\unit_\C$.
    \item 
    Given a Cartesian closed category $(\V, \times, \unit_{\V}),$ one can consider monoidal product $\otimes $ that is a subfunctor of $\times.$ This will automatically be monically projecting and semicartesian, and hence our assumptions will apply to operads in $\P$. This gives a rich source of examples from logic and theoretical computer science:
    \begin{enumerate}
        \item 
        Nominal sets with separated products \cite[Section 3.4]{pitts13}. 
        \item 
        Reflective graphs with box products \cite{Kapulkin24}.
        \item 
        Any commutative integral quantale $(Q, \ast, \top)$ is a semicartesian closed symmetric monoidal thin category \cite{flagg97}. Examples include complete Heyting algebras, the unit interval $[0,1]$ with a left-continuous $t$-norm, and the Lawvere quantale \cite{Lawvere73}. One can equip $Q$ with a semicartesian monoidal product $\cdot $ such that $a \cdot b  \le a \ast b$ for all $a,b \in Q$. This defines a subcategory.
        \item 
        The running example is \emph{almost} an example of this. The bifunctor $\vee$ is a subbifunctor of $\times$ in $\bTop.$
    \end{enumerate}
     It follows that categories of $\P$-coalgebras in the above categories are comonadic.
    \item
    Let $\C$ be a cocomplete, complete closed Cartesian category with terminal object $\ast$. Then the coslice category $\bC$ over $\ast$ has $\ast$ as a zero object. Denote the coproduct in this category by $\vee.$ Then for all $X, Y \in \bC,$ there is a map $X \vee Y \to X\times Y$ induced by the universal property of the product applied to the maps $\pi_1: X\vee Y \to X$ and $\pi_2: X\vee Y \to Y.$ In many cases of interest (topological spaces, Grothendieck toposes), this map is easily seen to be a monomorphism. The category $\bC$ is thus monically projecting as it is semicartesian, and therefore $\Pcoalg$ is comonadic over $\bC$ for any unital operad $\P$ in $\C$.  Similar arguments work in many related contexts, for example: one can consider pointed sheaves on a site with wedge product. In all the above cases all categories of coalgebras are comonadic.
    \item 
    Applying the previous example to $\Top$, we obtain the category $(\bTop, \vee)$ seen as  a $(\Top, \times)$-category. The little $n$-cubes operad $C_n$ is an operad in $\Top$ and we recover the comonad defined \cite[Section 2]{me1}.  The coalgebras over this operad are shown to all be homotopic to iterated suspensions. For a more detailed discussion, we refer the reader to that paper. 
\end{enumerate}

\appendix
\section{The axioms of symmetric monoidal $\V$-categories}
\label{appenda}
In the interests of being self-contained, we give the various axioms that symmetric monoidal $\V$-categories and functors between them must satisfy.

The \emph{associativity axiom} ensures that composition is associative up to the associator $\alpha$ of the monoidal category $\V$. For any objects $W, X, Y, Z \in \C$, the following diagram must commute:
\begin{center}
\begin{tikzcd}[column sep=huge, row sep=huge]
(\uC(Y, Z) \times \uC(X, Y)) \times \uC(W, X)
    \arrow[r, "\alpha"]
    \arrow[d, "c_{X,Y,Z} \times \id"']
& \uC(Y, Z) \times (\uC(X, Y) \times \uC(W, X))
    \arrow[d, "\id \times c_{W,X,Y}"] \\
\uC(X, Z) \times \uC(W, X)
    \arrow[r, "c_{W,X,Z}"']
& \uC(W, Z)
\end{tikzcd}
\end{center}

The \emph{unitality axioms} (left and right) ensure that composing with the identity morphism behaves neutrally, up to the left and right unitors ($\lambda$ and $\rho$) of $\V$. For any objects $X, Y \in \C$, the following two diagrams must commute:
\begin{center}
\begin{tikzcd}[column sep=large, row sep=large]
\unit_{\V} \times \uC(X, Y)
    \arrow[r, "j_Y \times \id"]
    \arrow[rd, "\lambda"']
& \uC(Y, Y) \times \uC(X, Y)
    \arrow[d, "c_{X,Y,Y}"] 
& \uC(X, Y) \times \unit_{\V}
    \arrow[r, "\id \times j_X"]
    \arrow[rd, "\rho"']
& \uC(X, Y) \times \uC(X, X)
    \arrow[d, "c_{X,X,Y}"] \\
& \uC(X, Y)
& 
& \uC(X, Y)
\end{tikzcd}
\end{center}
Our axioms for a $\V$-functor $F: \uC \to \uD$ are the following, which are just enriched variants of the usual axioms for a functor.
$$
\begin{tikzcd}
\uC(B, C) \times \uC(A, B) \arrow[r, "c"] \arrow[d, "F \otimes F"'] & \uC(A, C) \arrow[d, "F"] \\
\uD(FB, FC) \times \uD(FA, FB) \arrow[r, "c"'] & \uD(FA, FC),
\end{tikzcd}
$$

$$
\begin{tikzcd}[row sep=1.5em]
& \uC(A, A) \arrow[dd, "F"] \\
\unit_{\V} \arrow[ru, "j"] \arrow[rd, "j"'] & \\
& \uD(FA, FA).
\end{tikzcd}
$$
Finally, the  \emph{enriched naturality axiom} for a $\V$-natural transformation states the following: For every pair of objects $X, Y \in \uC$, the following diagram must commute in $\V$:
\begin{center}
\begin{tikzcd}[column sep=huge, row sep=huge]
\uC(X, Y) 
    \arrow[r, "F_{X,Y}"] 
    \arrow[d, "G_{X,Y}"'] 
& \uD(F(X), F(Y)) 
    \arrow[d, "(\alpha_Y)_*"] \\
\uD(G(X), G(Y)) 
    \arrow[r, "(\alpha_X)^*"'] 
& \uD(F(X), G(Y))
\end{tikzcd}
\end{center}
\section{Examples of symmetric monoidal $\V$-categories}
In this section, we give the details on our examples of monoidal $\V$-categories and explicitly describe the enrichment.

\begin{proposition}
\label{enrich}
    The symmetric monoidal category $(\bTop, \vee)$ with wedge sum is enriched over the category $\Top$ with Cartesian products $\times$.
\end{proposition}
\begin{proof}
   The category $\Top$ is closed symmetric monoidal \cite{steenrod67}.
   
   Now we may explicitly describe the enrichment of $\bTop$ over $\Top$. For any based spaces $\left(X, x_0\right)$ and $\left(Y, y_0\right)$, the enriched hom-object is the space of basepoint-preserving maps $\ubTop\left(X, Y\right)$, equipped with the subspace topology inherited from the unbased mapping space $\uTop\left(X, Y\right)$. Composition in $\bTop$ is simply the restriction of the continuous composition maps in $\Top$, making $\bTop$ a $\Top$-enriched category. 
   
   The wedge sum $\vee$ is the categorical coproduct in the underlying category $\left(\bTop\right)_0,$ so it follows that it defines a symmetric monoidal structure.

   Finally, we must show that the wedge sum $\vee$ defines the structure of a symmetric monoidal $\V$-category. It suffices to show that $\vee$ is an enriched coproduct, meaning it satisfies the universal property of the coproduct internally via a natural homeomorphism in $\bTop$:
   $$
   \vee_{(X, Y), (Z, Z')} \from \ubTop\left(X, Z\right) \times \ubTop\left(Y, Z'\right) \to  \ubTop\left(X \vee Y, Z\vee Z'\right)  
   $$
   This is the product of the inclusion of subspaces, quotiented by identifying basepoints.  Therefore, $\left(\bTop, \vee\right)$ is a symmetric monoidal $\Top$-category.  So it is continuous and clearly natural. The conclusion follows. The symmetric monoidal identities follow from the fact that $\vee$ is the categorical coproduct in the (unenriched) category $\bTop.$
\end{proof}

Another important source of examples are when one has a simplicial enrichment. Given an algebraic operad $\P$ (for precise details see \cite{loday12}), the category of $\P$-algebras can be equipped with an simplicially enriched structure as follows.
\begin{recollection}
\label{recol4}
The commutative operad is the unit for the Hadamard product $\otimes_H$ \cite[Section 5.3.3]{loday12}. Let $X\in \Comalg$ and $A,B\in \Palg$, then by \cite[Proposition 5.3.4]{loday12} the tensor product $X \otimes A$ is an algebra over $\Com\otimes_H \P \cong \P$. Explicitly, there is a natural isomorphism
$$
\Palg\left(A, X\otimes B\right) \cong \Com\otimes_H \P \mathsf{Alg}\left(X\otimes A, X\otimes B\right)
$$
\end{recollection}
We can establish the following.
\begin{proposition}
\label{propa5}
    The symmetric monoidal category $(\Palg, \Pcoprod)$, with categorical coproduct, is enriched over the category $\sSet$ with Cartesian products $\times$.
\end{proposition}
\begin{proof} 
    The category of simplicial sets equipped with Cartesian product $\times$ is a Grothendieck topos and therefore Cartesian closed. We define the simplicial mapping space 
    $$
    \uPalg(A, B)_n := \Palg(A,  A_{PL}^\ast(\Delta^n )\otimes B  )
    $$
    via Recollection \ref{recol4}, where the simplicial structure is induced by the simplicial structure on $A_{PL}^\ast(\Delta^n)$ (where $A_{PL}^\ast(\Delta^n)$ denotes the polynomial differential forms functor after Sullivan).

    To show that $\Pcoprod$ defines an \emph{enriched} symmetric monoidal structure, we must verify that it acts as an enriched bifunctor. Given $A, B, C, D \in \Palg$, we construct a map of simplicial sets:
$$
\Psi: \uPalg(A, C) \times \uPalg(B, D) \rightarrow \uPalg(A \Pcoprod B, C \Pcoprod D)
$$

We define $\Psi$ level-wise. In degree $n$, let $f \in \uPalg(A, C)_n$ and $g \in \uPalg(B, D)_n$. We have canonical inclusions $\iota_C: C \rightarrow C \Pcoprod D$ and $\iota_D: D \rightarrow C \Pcoprod D$. Since tensoring with the commutative algebra $A_{PL}^\ast(\triangle^n)$ preserves $\Palg$ morphisms, these induce maps:
$$
\text{id} \otimes \iota_C: A_{PL}^\ast(\triangle^n) \otimes C \rightarrow A_{PL}^\ast(\triangle^n) \otimes (C \Pcoprod D)
$$
$$
\text{id} \otimes \iota_D: A_{PL}^\ast(\triangle^n) \otimes D \rightarrow A_{PL}^\ast(\triangle^n) \otimes (C \Pcoprod D)
$$
By post-composing, we obtain $\tilde{f} = (\text{id} \otimes \iota_C) \circ f$ and $\tilde{g} = (\text{id} \otimes \iota_D) \circ g$. By the universal property of the categorical coproduct $A \Pcoprod B$, the maps $\tilde{f}$ and $\tilde{g}$ uniquely determine a single $\Palg$ morphism:
$$
f \Pcoprod g: A \Pcoprod B \rightarrow A_{PL}^\ast(\triangle^n) \otimes (C \Pcoprod D)
$$
We define $\Psi_n(f, g) := f \Pcoprod g$.  Compatibility with the simplicial structure is easily checked. 

We now verify that $(\Palg, \Pcoprod)$ assembles into a symmetric monoidal $\sSet$-category. First, the assignment $\Psi$ is functorial in each variable level-wise by the universal property of the coproduct, and the two functorialities commute because the coproduct is symmetric in its inputs; this is precisely the enriched bifunctor axiom. The unit isomorphism $\unit_\Palg \Pcoprod A \cong A$ and the associator $(A \Pcoprod B) \Pcoprod C \cong A \Pcoprod (B \Pcoprod C)$ are isomorphisms in the underlying category $\Palg_0$; tensoring with $A_{PL}^\ast(\Delta^n)$ preserves these isomorphisms by Recollection \ref{recol4}, so they lift level-wise to isomorphisms of simplicial mapping spaces. The symmetry $A \Pcoprod B \cong B \Pcoprod A$ is an isomorphism in $\Palg_0$ and lifts level-wise by the same argument. In each case, the lifted maps are natural in $A, B, C$ at every simplicial degree because they are induced by the universal property of the coproduct, which is natural; level-wise naturality in $\sSet$ is enriched naturality. The monoidal coherence axioms  hold level-wise because they hold in $\Palg_0$ and tensoring with $A_{PL}^\ast(\Delta^n)$ is a functor.
 Therefore, $(\Palg, \Pcoprod)$ is a symmetric monoidal $\sSet$-category.
\end{proof}
This immediately allows us to define $\sSet$-enriched coalgebras in $\Palg.$ Similarly, given a Koszul cooperad $\P$, the category $\Pcoalg$ of (non-enriched) coalgebras of $\P$ is Quillen equivalent to the category of algebras over the Koszul dual operad $\P^\antishriek.$ One can therefore apply the above construction to the category $\P^\antishriek\mathsf{Alg}$ to study morphisms, up to homotopy, between $\P$-coalgebras.

Let $\E$ be the Barratt-Eccles operad \cite{berger04}.
\begin{proposition}
    The symmetric monoidal category $(\Gcoalg, \oplus)$ is enriched over the category $\sSet$ with Cartesian products $\times$.
\end{proposition}
\begin{proof}
    The operad $\E$ is Hopf, meaning that there is a map $\E \to \E\otimes \E.$ This equips the tensor product of $\E$-coalgebras with a $\E$-coalgebra structure. We can thus define a simplicial enrichment on  $\Galg$ via 
    $$
    \uGcoalg(A, B)_n = \Gcoalg(A\otimes \SC(\triangle^n), B)
    $$
    From there the proof proceeds as in Proposition \ref{propa5}.
\end{proof}
\bibliographystyle{alpha}
\bibliography{MyBib}

\end{document}